\documentclass[12pt]{amsart}

\usepackage{graphicx, xcolor}
\usepackage{comment}
\usepackage{latexsym,amsmath,amsthm,amssymb,amsfonts,mathrsfs,MnSymbol,epsfig}
\usepackage[initials]{amsrefs}
\usepackage[all]{xy}

\usepackage{xcolor}

\usepackage{stackengine}
\stackMath

\providecommand{\CC}{{\mathbb{C}}}
\providecommand{\RR}{{\mathbb{R}}}

\providecommand{\ZZ}{{\mathbb{Z}}}

\providecommand{\KK}{{\mathcal K}}

\providecommand{\mcA}{{\mathcal{A}}}
\providecommand{\mcB}{{\mathcal{B}}}

\providecommand{\mcI}{{\mathcal{I}}}
\providecommand{\mcO}{{\mathcal{O}}}
\providecommand{\mcS}{{\mathcal{S}}}
\providecommand{\mcW}{{\mathcal{W}}}

\providecommand{\msL}{{\mathscr{L}}}
\providecommand{\msS}{{\mathscr{S}}}
\providecommand{\msW}{{\mathscr{W}}}

\providecommand{\Sch}{{\mathcal{S}}}

\providecommand{\Trh}{{\overline{\mathrm{Tr}}}}
\providecommand{\Alg}{{\mathscr{A}}}
\providecommand{\Symb}{{\mathscr{S}_H}}
\providecommand{\Symbm}{{\mathscr{S}^m}}
\providecommand{\Symball}{{\mathscr{S}^\bullet}}

\providecommand{\Sp}{{\mathrm{Sp}}}

\newcommand{\ang}[1]{\langle #1 \rangle} 
\newcommand{\lra}{\longrightarrow}

\DeclareMathOperator{\Hom}{Hom}
\DeclareMathOperator{\End}{End}
\DeclareMathOperator{\Ch}{Ch}

\DeclareMathOperator{\Td}{Td}
\DeclareMathOperator{\Tr}{Tr}
\DeclareMathOperator{\tr}{tr}
\DeclareMathOperator{\re}{Re} 
\DeclareMathOperator{\Res}{Res} 
\DeclareMathOperator{\Ker}{Ker} 

\DeclareMathOperator{\id}{id}
\DeclareMathOperator{\ind }{Index}

\DeclareMathOperator{\Sym}{Sym}
\DeclareMathOperator{\Op}{Op}

\newcommand{\con}{\boldsymbol{\nabla}}
\newcommand{\curv}{\boldsymbol{\theta}}
\newcommand{\ho}{\mathcal{H}}
\DeclareMathOperator{\iso}{\mu}
\DeclareMathOperator{\Iso}{\nu}

\newtheorem{theorem}{Theorem}[section]
\newtheorem{lemma}[theorem]{Lemma}
\newtheorem{corollary}[theorem]{Corollary}
\newtheorem{proposition}[theorem]{Proposition}

\theoremstyle{definition}
\newtheorem{definition}[theorem]{Definition}

\theoremstyle{remark}
\newtheorem{remark}[theorem]{Remark}
\newtheorem{example}[theorem]{Example}

\numberwithin{equation}{section}

\title{The Heisenberg Calculus, Index Theory and Cyclic (Co)homology}
\author{Alexander Gorokhovsky}
\author{Erik van Erp}

\begin{document}
\maketitle
\begin{abstract}
A hypoelliptic  operator in the Heisenberg calculus on a compact contact manifold is a  Fredholm operator.
Its symbol  determines an element  in the $K$-theory of the noncommutative algebra of Heisenberg symbols.
We construct a periodic cyclic cocycle which, when paired with the Connes-Chern character of the principal Heisenberg symbol,   calculates the index.
Our index formula is local, i.e. given as a local expression in terms of the principal symbol of the operator and a connection on $TM$ and its curvature.
We prove our index formula  by reduction to Boutet de Monvel's index theorem for Toeplitz operators.
\end{abstract}

\setcounter{tocdepth}{1}
\tableofcontents

\section{Introduction}

On  a compact contact manifold, a pseudodifferential operator in the Heisenberg calculus with an invertible symbol is a hypoelliptic Fredholm operator. The index problem for Heisenberg elliptic operators has been considered from various perspectives \cites{EM98, EMxx, vE10a, vE10b, BvE14, Po01}. Our approach here is most closely aligned with that of Epstein and Melrose in \cites{EM98, EMxx}. 
In this paper we give a local formula for the index,  as an expression in terms of the principal symbol of the operator and a connection on $TM$ and its curvature.

\vskip 6pt

\subsection{Heisenberg elliptic operators} 
Let $M$ be a smooth manifold of odd dimension $2n+1$.
A contact form on $M$ is a differential 1-form  $\alpha$ such that $\alpha(d\alpha)^n$ is a nowhere vanishing volume form.
The Reeb  field $T$ is the vector field on $M$ with $\alpha(T)=1$, $d\alpha(T,\,\cdot\,)=0$.
In \cite{FS74}, Folland and Stein  showed how, for the analysis of certain naturally occurring hypoelliptic operators in the theory of several complex variables,  the Reeb field should be treated as a differential operator of order 2.
Only vector fields $X$ with $\alpha(X)=0$ are given order 1.
This filtration of the Lie algebra of vector fields determines a non-standard  filtration on the algebra of differential operators on $M$.
The `highest order part' of an operator $\msL$ on $M$ is realized as a smooth family $\msL_p$, parametrized by $p\in M$,
of translation invariant operators on the Heisenberg group.
In \cite{FS74}, fundamental solutions for the model operators $\msL_p$ are utilized to construct a parametrix for  $\msL$.

The ideas of Folland and Stein are the root of  the Heisenberg pseudodifferential calculus \cites{Ta84, BG88}.
The Heisenberg algebra of a contact manifold is a $\ZZ$-filtered algebra $\Psi_H^\bullet$ of pseudodifferential operators on $M$ in which the Reeb field is an operator of order 2.
If $M$ is compact, operators of Heisenberg order zero $\Psi_H^0$ are bounded on $L^2(M)$,
and operators of negative order $\Psi^{-1}_H$ are compact.
We denote the algebra of principal Heisenberg symbols of order zero by
\[ \Symb := \Psi^0_H/\Psi^{-1}_H\]
$\Symb$ is a noncommutative  algebra. 

The algebra of Heisenberg symbols is as follows.
Let $H=\Ker \alpha$ be the vector bundle of tangent vectors that are annihilated by the contact form $\alpha$.
The restriction of the 2-form $\omega:=-d\alpha$ to a fiber of $H$  is a symplectic form.
Associated to the symplectic bundle $H$ is a bundle of Weyl algebras over $M$, which we denote $\mcW_H$.
The fiber of $\mcW_H$ at $p\in M$ consists of smooth functions on $H_p^*$
that have an asymptotic expansion, modulo Schwartz class functions, as a sum of homogeneous terms of integral order.
The product in the Weyl algebra is such that for two linear functions $f, g$ on $H^*$ (i.e. vectors in $H_p$) the commutator is
\[ f\#g-g\#f = i\omega(f,g)\]
In general,
\begin{equation}\label{sharp product}
    (f\# g)(v)  = \frac{1}{(2\pi)^{2n}}\int_{H_p} e^{2i\,\omega(x,y)}f(v+x)g(v+y) dx\,dy
\end{equation} 
We denote by $\mcW^{op}_H$ the bundle of algebras with the opposite product, i.e. the product of $f, g$ in $\mcW_H^{op}$ is $g\#f$.
The principal Heisenberg symbol $\sigma_H(T)$ of a Heisenberg pseudodifferential operator on $M$
is a pair of smooth functions on $H^*$,
\[ \sigma_H(T)=(\sigma_+,\sigma_-)\in \mcW_H\oplus \mcW_H^{op}\]
(See section \ref{sec:symbols} below.)
$T$ is Heisenberg elliptic if  $(\sigma_+, \sigma_-)$ is invertible in $\mcW_H\oplus \mcW_H^{op}$.
A Heisenberg elliptic operator $T$ has a parametrix (an inverse modulo smoothing operators),
and if $M$ is compact then $T$  is a  Fredholm operator.

\subsection{Index formulas of Epstein and Melrose} 
In this paper we present a formula for the index of a general Heisenberg elliptic operator on a compact contact manifold that  is computable from local data.
The impetus for our paper is the fundamental work of Charles Epstein and Richard Melrose from the late 1990s \cites{EM98, Ep04}. 
We review some  key points of their work.

Epstein and Melrose  derive an index formula for what they call {\em Hermite operators}.
Denote by $\Sch(H^*)$ the algebra of smooth functions on $H^*$ that are Schwartz class in each fiber, with the $\#$ product \eqref{sharp product}.
$P\in \Psi^0_H$ is a Hermite operator if
\[ \sigma^0_H(P)=(\sigma_+,\sigma_-)\in \Sch(H^*)\oplus\Sch(H^*)^{op}\]
A Toeplitz operator is an example of a Hermite operator.
The following formula of Epstein-Melrose \cite{Ep04} generalizes Boutet de Monvel's index formula for Toeplitz operators \cite{Bo79}. If $P$ is Heisenberg elliptic, and  $P-1$ is a Hermite operator, then
\begin{equation}\label{hermite index}
     \ind P=\int_M\Ch(\sigma_+)\wedge \Td(H^{1,0})+(-1)^{n+1}\int_M \Ch(\sigma_-)\wedge \Td(H^{0,1})
\end{equation}
Here $H\otimes_\RR \CC = H^{1,0}\oplus H^{0,1}$ is the splitting of the complexified vector bundle $H\otimes_\RR \CC$
into the $i$ and $-i$ eigenspaces of a complex structure $J\in \End H$, $J^2=-\mathrm{Id}_H$.
Such a complex structure is obtained by reducing the structure group of $H$ from the symplectic group $\mathrm{Sp}(2n)$ to the unitary group $U(n)$.

The definition of the Chern character appearing in \eqref{hermite index} requires some discussion.
Each fiber  of $\Sch(H^*)$
can be represented faithfully by trace class  operators on a Hilbert space
(the Bargmann-Fock space).
This Hilbert space  is obtained by  completion of the symmetric tensors of the complex vector space  $H_p^{1,0}$,
\[ \Sym H^{1,0}=\bigoplus_{j=0}^\infty \Sym^j H^{1,0}\]
For sufficiently large $N$, the compression $\sigma_+^{(N)}$ of $\sigma_+$ to the finite rank vector bundle 
\[\Sym^{(N)} H^{1,0}=\bigoplus_{j=0}^N \Sym^j H^{1,0}\]
is  invertible. 
Thus, $\sigma_+^{(N)}$ is an automorphism of the vector bundle $\Sym^{(N)}H^{1,0}$,
and  defines an element in   $K^1(M)$. 
For large enough $N$, this class in $K$-theory is constant, and 
the cohomology class of $\Ch(\sigma_+)$ is
\begin{equation}\label{chern_chomology}
    [\Ch(\sigma_+)]:=[\Ch(\sigma_+^{(N)})]\in H^{odd}(M)\qquad N>>0
\end{equation} 
A Chern-Weil formula for the differential form $\Ch(\sigma_+)$
is obtained by suspending $\sigma_+^{N}$.
This results in  a vector bundle $E(\sigma_+^{(N)})$ on $M\times S^1$.
One can write an explicit formula for the curvature $\Omega(\sigma_+^{(N)})$ of $E(\sigma_+^{(N)})$.
This formula  depends on a choice of connection for $H^{1,0}$.
Then \eqref{hermite index} is made explicit by using the formulas for $\Omega(\sigma_+^{(N)})$ and taking the limit
\begin{equation}\label{chern odd}
    \Ch(\sigma_+):=  \lim_{N\to \infty} \Tr(\exp(\frac{i}{2\pi}\Omega(\sigma_+^{(N)}))
\end{equation} 
The Chern character  $\Ch(\sigma_-)$ is treated similarly,
except that now $\Sch(H^*)\subset \mcW_H^{op}$ is represented by operators on $\Sym H^{0,1}$.

Epstein and Melrose  derive index formulas for more general  (but not quite all) Heisenberg elliptic operators.
These formulas are less explicit and  more complicated than \eqref{hermite index}.
As Epstein explains in \cite{Ep04}, there are two problems that make it difficult to further generalize \eqref{hermite index}.
The first problem arises if  $\sigma_+-1$ and $\sigma_--1$ converge to  $0$ at infinity, but are not  integrable.
In this case the trace used in the computation of the Chern character no longer converges as $N\to \infty$.
A regularization  of the trace, using zeta function methods,
 results in additional  terms in the index formula.
The second obstacle is encountered in the most general case, if the symbol $(\sigma_+,\sigma_-)$
does not converge to $1$ at infinity. 
In concrete  examples one may obtain index formulas by  reducing to the Hermite case using $K$-theoretic arguments, but no  formula is obtained in this way that applies generally. The interested reader is referred to the   monograph \cite{EMxx}.

\subsection{$K$-theoretic index theorems}

In \cite{vE10a}, the second author adapted  Alain Connes' use of the tangent groupoid   to the index problem for Heisenberg elliptic operators. 
The symbol of a Heisenberg elliptic operator determines, in a canonical way, an idempotent in a noncommutative $C^*$-algebra $M_2(C^*(T_HM))$, which gives an element in $K$-theory.
$T_HM$ is a bundle of Heisenberg groups over $M$.
The  group $K_0(C^*(T_HM))$ is isomorphic (by general results in noncommutative geometry) to $K^0(T^*M)$,
and the Atiyah-Singer index map $K^0(T^*M)\to \ZZ$ is shown to compute the index also in the case of Heisenberg elliptic operators.

It is a highly nontrivial problem  to derive locally computable index formulas from this $K$-theoretic result, because the isomorphism $K_0(C^*(T_HM)\cong K^0(T^*M)$ is not given by an explicit formula.
In \cite{vE10b}, an index formula is derived 
for {\em scalar} operators $P\in\Psi_H^0$,
\begin{equation}\label{scalar}
    \ind  P =\int_M \Ch(\sigma_+\#(\sigma_-)^{-1})\wedge \Td(H^{1,0})
\end{equation} 
It should be mentioned that $\sigma_+\#(\sigma_-)^{-1}$ converges to $1$ at infinity, but not rapidly.
So the cohomological definition \eqref{chern_chomology} applies, while the limit in \eqref{chern odd} does not converge.
A more severe limitation of \eqref{scalar}  is that this formula is {\em incorrect}   for operators that act on sections in vector bundles. 
In many familiar examples, an index formula for scalar operators is valid, with minor changes, for operators acting on vector bundles.
However, this is not the case here.
For an example illustrating this point, see Example 6.5.3 in \cite {BvE14}.

In   \cite{BvE14} the problem is approached in the context of $K$-homology.
A Baum-Douglas geometric $K$-cycle is constructed that corresponds to the analytic $K$-cycle (in $KK$-theory)  determined by a Heisenberg elliptic operator.
In the case of Hermite operators one can  extract a characteristic class formula from this result,
which turns out to be equivalent to \eqref{hermite index}.
For the general case,  
the computation of the $K$-cycle proceeds by showing that every Heisenberg elliptic operator is equivalent, in $K$-homology,
to a Hermite operator.
However, this equivalence in $K$-homology does not resolve the problem of deriving a locally computable index formula in the general case.

Taken together \cites{vE10a, vE10b, BvE14} provide a satisfactory solution of the index problem in the context of $K$-theory. Nonetheless, the problem of finding concrete index formulas was not  advanced much beyond the work of Epstein-Melrose.
For the Heisenberg calculus, the transition from $K$-theoretic index theorems to locally  computable index {\em formulas}  is highly non-trivial.

\subsection{Our  index formula}
As a first step towards a resolution of the difficulties listed by Epstein in \cite{Ep04},
we simplify the formula \eqref{chern odd} of Epstein-Melrose for the odd Chern character.
Our formula for the odd Chern character is taken from  the  thesis of  the first author  \cites{gor1, gor2}.
If  $\sigma_+-1\in \Sch(H^*)$ and $\sigma_+$ is invertible, our formula for the Chern character is
\begin{equation}\label{chern odd new}
\Ch(\sigma_+) = \sum \limits_{l\ge 0} \sum \limits_{i_0, \ldots, i_{2k+1} \ge 0} \left(\frac{-1}{2 \pi i}\right)^{I+l+1} \frac{l!}{(I+2l+1)!} 
\Tr \left( 
\sigma_+^{-1} \theta^{i_0}
\nabla(\sigma_+)\theta^{i_1}\nabla(\sigma_+^{-1}) \dots \nabla(\sigma_+) \theta^{i_{2l+1}}
\right)
\end{equation}
\[
I=i_0+i_1+\cdots+i_{2l+1}
\]
(See Proposition \ref{chern} below.)
The connection $\nabla$ is determined by choosing a unitary connection on $H^{1,0}$,
which determines a connection on the symmetric tensor bundle $\Sym H^{1,0}$.
Then $\nabla$ also determines a connection on   $\Sch(H^*)$, thought of as smooth families of operators on $\Sym H^{1,0}$.
The curvature of $\nabla$ is a 2-form $\theta$ which acts  on the fibers of $\Sym H^{1,0}$.
We may think of $\theta$ as a multiplier of $\Sch(H^*)$, and for $a\in \Sch(H^*)$ we have
\[ \nabla^2(\sigma_+)=[\theta,\sigma_+]\]
In section \ref{section_curvature} below we construct a curvature 2-form $\curv$ which is a multiplier of the combined algebra $\Sch(H^*)\oplus \Sch(H^*)^{op}$, but which differs  from the $\End(\Sym H^{1,0})$-valued curvature 2-form used in the Chern character formula \eqref{chern odd new} by 
a {\em scalar} 2-form.
This subtle change makes it possible to combine the two summands of \eqref{hermite index} into a single formula,
\begin{equation}\label{general index}
     \ind P=\int_M \chi(\sigma_H(P))\wedge \hat{A}(M)
\end{equation}
where, with $\sigma_H(P)=\sigma=(\sigma_+,\sigma_-)$,
 \begin{equation}\label{chi_character}
\chi(\sigma) = \sum \limits_{l\ge 0} \sum \limits_{i_0, \ldots, i_{2k+1} \ge 0} \left(\frac{-1}{2 \pi i}\right)^{I+l+1} \frac{l!}{(I+2l+1)!} \tau \left(  \sigma^{-1} \curv^{i_0} \con(\sigma)\curv^{i_1}\con(\sigma^{-1}) \dots \con(\sigma) \curv^{i_{2l+1}}\right)
\end{equation}
\[
I=i_0+i_1+\cdots+i_{2l+1}
\]
and $\tau$ is the combined trace
\[ \tau(a_+,a_-):=\Tr(a_+)+(-1)^{n+1}\Tr(a_-)\qquad (a_+,a_-)\in \Sch(H^*)\oplus \Sch(H^*)^{op}\] 
Note that we replaced the two Todd classes by $\hat{A}(M)$, which allowed us to combine the terms.
As our computations in section \ref{section_Toeplitz} will show, the exponential factors in
\[ \Td(H^{1,0})=\hat{A}(M)\wedge \exp(\frac{1}{2}c_1(H^{1,0}))\qquad \Td(H^{0,1})=\hat{A}(M)\wedge \exp(-\frac{1}{2}c_1(H^{1,0}))\]
are absorbed in $\chi(\sigma)$ by our choice of $\curv$.

The  main result of our paper is that  formula \eqref{general index} applies {\em without change} to general Heisenberg elliptic operators. 
First, the trace $\tau$ extends to a trace on the full algebra of principal Heisenberg symbols $\Symb$,
\[ \tau(\sigma_+,\sigma_-):=\Trh(\sigma_+)+(-1)^{n+1}\Trh(\sigma_-)\qquad (\sigma_+,\sigma_-)\in \Symb\]
where $\Trh$ is the regularized trace of \cite{EMxx}.
While $\Trh$ is not a trace, $\tau$ is a trace on $\Symb$.
Note that, despite the use of the regularized trace $\Trh$,  our general formula does not require any extra residue terms.
Secondly, the curvature form $\curv$ is not a multiplier of $\Symb$.
We  introduce an algebra $\Alg_H$, larger than $\Symb$ but smaller than $\mcW_H\oplus \mcW^{op}_H$,
\[ \Symb\subset \Alg_H\subset \mcW_H\oplus \mcW_H^{op}\]
The algebra $\Alg_H$ is large enough so that $\curv\in \Omega^2(\Alg_H)$, and small enough so that the trace $\tau$  extends to $\Alg_H$.
Interpreted in this way,  our index formula  \eqref{general index} is valid in the general case.
\vskip 6pt
\noindent{\it Remark.}
$\chi(\sigma_H(P))$ is a closed differential form,
and its cohomology class is independent of the choice of connection on $H$.
In fact, $\chi$ determines a homomorphism,
\[ \chi:K_1(\Symb)\to H^{odd}(M)\]


\subsection{Characters and cycles}
Cyclic (co)homology plays a fundamental role in our approach.
Our formula is obtained from a cyclic cocycle for the symbol algebra $\Symb$. The ingredients of this cocycle are as follows.

\begin{itemize}
\item In section \ref{sec:alg} we construct an algebra $\Alg_H$ that contains  $\Symb$ as a subalgebra.
$\Alg_H$ consists of smooth sections in a bundle of algebras 
associated to $H$,
\[ \mathcal{A}_H := P_H\times_{\mathrm{Sp}(2n)} \Alg\]
Here $P_H$ is the principal $\mathrm{Sp}(2n)$ bundle of symplectic frames of $H$, and  $\Alg$ is the model algebra for the fibers of $\Alg_H$. 
\item 
A symplectic  connection $\nabla$ for $P_H$ determines a connection $\con$ for $\Alg_H$.
This connection extends, in the usual way, to a connection on $\Alg_H$-valued differential forms,
\[ \con: \Omega^k(\Alg_H)\to \Omega^{k+1}(\Alg_H)\]
\item
In section \ref{section_curvature} we show that the curvature of the  connection $\con$ is inner, i.e. there  is an element
\begin{equation*} \curv \in \Omega^2(\Alg_H)\end{equation*}
such that 
\begin{equation*} \con^2(a) = [\curv, a]\end{equation*}
We show that $\con(\curv)=0$.
Thus, the triple $(\Omega(\Alg_H),\con,\curv)$ is a curved dga.
\item In section \ref{section_trace} we construct a graded trace
\begin{equation*} \tau:\Omega(\Alg_H) \to \Omega(M)\qquad \tau(ab) = (-1)^{|a|\,|b|}\tau(ba)\end{equation*}
which satisfies $\tau(\con(\beta))=d\tau(\beta)$.
\end{itemize}

In section \ref{section_character} we explain how to combine these data to obtain a cyclic cocycle over $\Symb$.
If we let 
\[ \stackinset{c}{}{c}{}{-\mkern4mu}\int \beta := \int_M \tau(\beta) \wedge 
\hat{A}(M)\qquad \beta\in \Omega(\Alg_H)\]
then the quadruple 
\[(\Omega(\Alg_H),
\con,\curv,\displaystyle\stackinset{c}{}{c}{}{-\mkern4mu}\int)\]
is a generalized cycle  (as defined in \cites{gor1, gor2}) over $\Symb$ (or, more correctly, a finite sum of generalized cycles).
The character of this generalized cycle is a periodic cyclic cocycle in $ HC_{per}^1(\Symb)$.
We show that this cocycle is continuous.
Thus we can pair it with topological $K$-theory, to  obtain a map
\[ K_1(\Symb)\stackrel{\Ch}{\lra} HC^{per}_1(\Symb)\to \CC\]  
The formulas developed by the first author in \cites{gor1, gor2} for the character of a generalized cycle  yield   our index formula \eqref{general index}.
\begin{theorem}
Let $M$ be a compact smooth manifold of dimension $2n+1$ with contact form $\alpha$. 
We orient  $M$ by the volume form $\alpha(d\alpha)^n$. 
If
\begin{equation*}P:C^\infty(M,\CC^r)\to C^\infty(M,\CC^r)\end{equation*}
is a Heisenberg pseudodifferential operator of order $m\in \ZZ$
that acts on sections in a trivial bundle $M\times \CC^r$, 
with invertible Heisenberg principal  symbol $\sigma^m_H(P)\in M_r(\Symbm)$, then the index of $P$ is
\begin{equation*} \ind  P = \int_M \chi(\sigma^m_H(P))\wedge \hat{A}(M)\end{equation*}
With $\sigma=\sigma^m_H(P)$ the character $\chi(\sigma^m_H(P))$ is as in \eqref{chi_character},
where the connection $\con$ and curvature $\curv$ are as in \eqref{defcon}, \eqref{deftheta}.
\end{theorem}
In section \ref{section_Toeplitz}, a calculation shows that  in the case of Toeplitz operators our formula reduces to that of Boutet de Monvel \cite{Bo79}.
These same calculations  imply that in the case of Hermite operators \eqref{general index} is equivalent to  \eqref{hermite index}, with the changes discussed above.

To prove that our formula holds in the general case, we use the fact that periodic cyclic homology is homotopy invariant.
It  suffices to prove that our formula   is correct for each equivalence class in $K_1(\Symb)$.
In section \ref{section_Ktheory} we show that the group $K_1(\Symb)$ is generated by vector bundle automorphisms and Toeplitz operators.
In section \ref{section_proof} we prove our index formula.
All that remains to show is that for a vector bundle automorphism
our  formula evaluates to zero.
Perhaps surprisingly, this turns out to be non-trivial, but reduces to proving that
\[ \tau(\curv^k)=0\qquad k=0,1,2,\dots\]
(Lemma \ref{tau_of_curv}).

\subsection*{Acknowledgements}
An important motivation  for our paper is the work of Charles Epstein and Richard Melrose, especially the monograph \cite{EMxx}.
We thank them for sharing their unpublished manuscript, and for several enlightening conversations.
We benefited greatly from conversations with Ryszard Nest, and thank him for his hospitality.

\section{Heisenberg principal symbols}\label{sec:symbols}

In this section we describe the algebra of principal symbols in the Heisenberg  calculus.
This section   serves to fix our notations.
We include a brief  description of Heisenberg symbols of differential operators for the sake of the reader unfamiliar with this calculus.
For  details on the Heisenberg pseudodifferential calculus, see \cites{Ta84, BG88, CGGP92, Ep04}.

\subsection{Contact structures}\label{sec:contact}

Throughout this paper, $M$ is a smooth closed manifold of dimension $2n+1$.
A contact form on $M$ is a 1-form $\alpha$ such that $\alpha(d\alpha)^n$ is a nowhere vanishing volume form.
The Reeb field $T$ is the vector field on $M$ determined by 
\begin{equation*} d\alpha(T,\,-\,) = 0\qquad \alpha(T)=1\end{equation*}
Let $H\subset TM$ be the hyperplane bundle of tangent vectors that are annihilated by $\alpha$.
Then 
 \[ TM\cong H\oplus \underline{\RR}\]
The restriction of the 2-form $\omega:=-d\alpha$ to each fiber $H_p, p\in M$ is a symplectic form,
\begin{equation*}\omega_p(v,w):=-d\alpha(v,w) \qquad v,w\in H_p\end{equation*}
We denote by  $P_H$  the $\mathrm{Sp}(2n)$ principal bundle of symplectic frames in $H$. 
 $P_H$ is equipped with a right action of $\mathrm{Sp}(2n)$.

For certain purposes we shall fix a complex structure $J\in \mathrm{End}(H)$, $J^2=-\mathrm{Id}$,
that is compatible with the symplectic structure, i.e.,
\begin{equation*}\omega(Jv,Jw) = \omega(v,w)\qquad \omega(Jv, v)\ge 0\end{equation*}
Let $H\otimes \CC = H^{1,0}\oplus H^{0,1}$ be the  decomposition of the fibers of $H$ into  $i$ and $-i$ eigenspaces of $J$.
$H^{1,0}$ is a hermitian vector bundle with hermitian form,
\[\ang{v,w}:=\omega(v,\bar{w})\]
We identify the complex vector bundle $H^{1,0}$ with the real vector bundle $H$ 
via $v\mapsto v+\bar{v}$. 
Thus, the structure group of $M$ can be further reduced to $U(n)\subset \mathrm{Sp}(2n)$.

\subsection{Osculating Lie algebras}
On a contact manifold $M$, the symplectic form $\omega_p$
on $H_p$ gives the fibers of the vector bundle
\[ \mathfrak{t}_HM := H\oplus TM/H\]
the structure of a 2-step graded nilpotent  Lie algebra, isomorphic to the Heisenberg Lie algebra.
We trivialize the normal bundle $TM/H$ by the Reeb field. 
Then the fiber at $p\in M$,
\begin{equation*} \mathfrak{t}_HM_p = H_p\oplus \RR T(p)\end{equation*}
is a Lie algebra in which the degree 2 element  $T(p)$ is central, and
\begin{equation*} [v,w]=\omega_p(v,w)T(p)\qquad v,w\in H_p\end{equation*}
Note that if $V,W\in \Gamma(H)$ are two vector fields tangent to $H$, then $\alpha(V)=\alpha(W)=0$ implies
\begin{equation*} \alpha([V,W])=-d\alpha(V,W)\end{equation*}
Therefore 
\[ [V(p),W(p)]=[V,W](p)\,\in T_pM/ H_p\]
where the left hand side is the bracket in $\mathfrak{t}_HM_p$, while 
 $[V,W]$ is the commutator of vector fields.

For each $s>0$ we have an Lie algebra automorphism of $\mathfrak{t}_HM_p$ by parabolic dilation,
\begin{equation*} \delta_s(v,t) := (sv, s^2t)\qquad (v,t)\in H_p\times \RR=\mathfrak{t}_HM_p\end{equation*}

\subsection{Heisenberg symbols  of differential operators}\label{sec:heisenberg diff op}

Let $\mathscr{P}^\bullet$ be the algebra of scalar differential operators on $M$.
The Heisenberg filtration on $\mathscr{P}^\bullet$
is determined by the requirement that vector fields $X\in \Gamma(H)$ are order one operators (as usual),
but the Reeb field $T$ is an order two operator.
Smooth functions $C^\infty(M)$ are central in the associated graded algebra
\[ \mathscr{P}_\bullet = \bigoplus_{k=0}^\infty \mathscr{P}_k \quad \mathscr{P}_k := \mathscr{P}^k/\mathscr{P}^{k-1}\qquad \mathscr{P}^{-1}=0\]
Thus, $\mathscr{P}_\bullet$ is the algebra of smooth sections in a bundle of algebras over $M$.

Denote by $\mathscr{U}(\mathfrak{t}_HM)_p$ the universal enveloping algebra of the Heisenberg Lie algebra $(\mathfrak{t}_HM)_p$, and let $\mathscr{U}(\mathfrak{t}_HM)$ be the bundle over $M$ with fiber $\mathscr{U}(\mathfrak{t}_HM)_p$.
Since  $\mathfrak{t}_HM_p$ is a graded Lie algebra, the universal enveloping algebra $\mathscr{U}_\bullet(\mathfrak{t}_HM)_p$ is a $\ZZ$-graded algebra.

The associated graded algebra $\mathscr{P}_\bullet$ is naturally identified (as a  graded algebra) with the algebra of smooth sections in $\mathscr{U}_\bullet(\mathfrak{t}_HM)$,
\[ \mathscr{P}_\bullet =C^\infty(M;\mathscr{U}_\bullet(\mathfrak{t}_HM))\]
Let $\mathscr{W}_p$ be the (polynomial) Weyl algebra, generated by vectors in $H_p$
with  relations
\[ vw-wv = i\omega_p(v,w)\qquad v,w\in H_p\]
There is a canonical identification of vector spaces given by symmetrization 
\[ \mathscr{W}_p \cong \Sym H_p\quad \sum_{s\in S_k} v_{s(1)}v_{s(2)}\,\cdots\,v_{s(k)}\mapsto \sum_{s\in S_k} v_{s(1)}\otimes v_{s(2)}\otimes \cdots\otimes  v_{s(k)}\]
In this way elements in $\mathscr{W}_p$ are identified with polynomials on $H^*_p$.
If we view  polynomials on $H^*_p$ as elements in $\msW_p$, we write the product  as $\#$ (see \eqref{sharp product}).

Let $\pi$ be the algebra homomorphism 
\[ \pi \colon \mathscr{U}(\mathfrak{t}_HM)_p\to \mathscr{W}_p\]
with $\pi(v)=v$ for all $v\in H_p$, and
\[  \pi(T(p)) = i\]
There is also a  homomorphism to the opposite of the Weyl algebra
\[ \pi^{op} \colon \mathscr{U}(\mathfrak{t}_HM)_p\to \mathscr{W}_p^{op}\]
with $\pi(v)=v$ for $v\in H_p$ and
\[ \pi^{op}(T(p)) = -i\]
Let $\mathscr{W}_H$ denote the bundle of Weyl algebras $\mathscr{W}_p$ over $M$.
For differential operators, the principal Heisenberg symbol  of degree $m$ is the algebra homorphism 
\[ \sigma_H^m \colon \mathscr{P}^m\to  \mathscr{W}_H\oplus \mathscr{W}_H^{op}\]
obtained as the composition 
\[ \mathscr{P}^m\to \mathscr{P}_m=C^\infty(M;\mathscr{U}_m(\mathfrak{t}_HM))\stackrel{(\pi,\pi^{op})}{\lra} \mathscr{W}_H\oplus \mathscr{W}_H^{op}\]
A vector field $X\in \Gamma(H)$ that is tangent to $H$ is a differential operator of Heisenberg order 1.
Its principal symbol is
\[\sigma^1_H(X)=(X,X)\in \msW_H\oplus \msW_H^{op}\]
At each point $p\in M$, the vector $X(p)\in H_p$ is interpreted as a linear function on $H_p^*$.
The Reeb vector field $T$ is a differential operator of Heisenberg order 2.
It's principal symbol is
\[\sigma^2_H(T)=(i,-i)\in \msW_H\oplus \msW_H^{op}\]
If $\mathscr{L}_1$, $\mathscr{L}_2$ are two differential operators of Heisenberg order $m_1$, $m_2$ respectively, then
\[ \sigma_H^{m_1+m_2}(\mathscr{L}_1\mathscr{L}_2) = \sigma_H^{m_1}(\mathscr{L}_1)\sigma_H^{m_2}(\mathscr{L}_2)\]
This follows easily from the fact that for any two vector fields $X, Y\in \Gamma(H)$
the bracket in $\mathfrak{t}_HM$ is defined precisely so that
\[ [\sigma_H^1(X),\sigma_H^1(Y)]=\sigma^2_H([X,Y])\]
Let $\sigma_H^m(\mathscr{L})=(\sigma_+,\sigma_-)$ where $\sigma_+(p)$ and $\sigma_-(p)$ are polynomials  on $H^*_p$.
The highest  order part $[\mathscr{L}]\in \mathscr{P}_m$ is at each point $p\in M$ an element $[\mathscr{L}](p)\in \mathscr{U}_m(\mathfrak{t}_HM_p)$ that is homogeneous of degree $m$.
Thus $[\mathscr{L}](p)$ is a sum of monomials $v_1v_2\cdots v_kT(p)^l$ with $k+2l=m$.
Then $\sigma_+(p)=\pi([\mathscr{L}](p))\in \mathscr{W}_p$ is a polynomial on $H_p^*$ that is  a  sum 
\[ \sigma_+ = \sum_{0\le l\le m/2} \sigma_{m-2l}\]
where $\sigma_{m-2l}(p)$ is homogenoeus of degree $m-2l$.
In other words, $\sigma_+(p)$ is either  an even or odd polynomial  on $H_p^*$.
Moreover, $\sigma_-$ can be recovered from $\sigma_+$ by
\[ \sigma_- = \sum_{0\le l\le m/2} (-1)^l \sigma_{m-2l}\]

\begin{example}
Choose  a compatible complex structure $J$ for $(H,\omega)$.
This  determines a  Euclidean structure for $H$.
Locally, in an open set where $H$ can be trivialized, 
let $X_1, X_2, \dots, X_{2n}$ be vector fields on $M$ that form an orthonormal frame for $H$.
A sublaplacian $\Delta_H$ on $M$ is 
\begin{equation*} \Delta_H = X_1^2+X_2^2+ \cdots+X_{2n}^2\end{equation*}
A global sublaplacian $\Delta_H$ is obtained by a partition of unity.
A different choice of orthonormal frame 
results in a sublaplacian that differs by a vector field that is tangent to $H$.
The Heisenberg principal symbol of $\Delta_H$ is independent of choices.
In the algebra $\msW_H$,
\[X_1\#X_1+X_2\#X_2+ \cdots+X_{2n}\#X_{2n}=Q\]
where $Q$ is the smooth function on $H^*$ which in each fiber is the quadratic polynomial 
\[ Q(v)=\|v\|^2\qquad v\in H_p^*\]
In the Weyl algebra, $Q$ is the harmonic oscillator.
We see that 
\[ \sigma_H^2(\Delta_H) = (Q,Q)\in \msW_H\oplus \msW_H^{op}\]
An example of a Heisenberg elliptic differential operator is as follows.
If $\gamma \colon M\to \CC$ is a smooth function, then the principal Heisenberg symbol of the operator
\[ \mathscr{L}_\gamma = \Delta_H+i\gamma T\]
is
\[ \sigma^2_H(\mathscr{L}_\gamma) = (Q-\gamma,Q+\gamma)\in \msW_H\oplus \msW_H^{op}\]
The operator $\mathscr{L}_\gamma$ is Heisenberg elliptic
if its  principal Heisenberg symbol is invertible
(in the larger algebra $\mcW_H\oplus \mcW_H^{op}$).
The spectrum of the harmonic oscillator is the set 
\[\Lambda=\{n, n+2, n+4,n+6,\dots\}\]
Thus, $Q-\gamma(p)\in \mcW(H^*_P,\omega^*_p)$ is invertible if and only if $\gamma(p)\nin \Lambda$,
while $Q+\gamma(p)\in \mcW(H^*_p,\omega^*_p)^{op}$ is invertible
if and only if $-\gamma(p)\nin \Lambda$.
The operator  $\mathscr{L}_\gamma$ is Heisenberg elliptic 
if and only if the coefficient $\gamma$
does not take values in the set $\Lambda\cup -\Lambda$.

\end{example}

\subsection{Heisenberg symbols of pseudodifferential operators}\label{sec:Hsymb}
The Heisenberg calculus on a contact manifold is a $\ZZ$-graded algebra $\Psi^\bullet_H$ of pseudodifferential operators.
We denote the set of Heisenberg principal symbols of order $m$ as $\Symbm$,
\[ 0\to \Psi_H^{m-1}\to \Psi_H^m\stackrel{\sigma_H^m}{\lra} \Symbm\to 0\]
For a Heisenberg pseudodifferential operator of degree $m$,
the principal symbol $\sigma=\sigma^m_H(T)$ is a smooth function on $(H^*_p\times \RR)\setminus \{(0,0)\}$
that is homogeneous of degree $m$ for the parabolic dilations $\delta$, 
\[ \sigma(sv,s^2t) = s^m\sigma(v,t)\qquad s>0,\;(v,t)\in (H^*_p\times \RR)\setminus \{(0,0)\}\]
We restrict the function $\sigma^m_H(T)$  to the hyperplanes 
\[ H^*_p\times \{t\} \subset H_p^*\times \RR\]
Let $\omega_p^*$ be the symplectic form on $H^*_p$ that is dual to $\omega_p$.
The product of Heisenberg symbols is the product $\#_t$ of elements in the Weyl algebra $\mcW(H_p^*,t\omega^*_p)$ in each hyperplane with $t\ne 0$ (see section \ref{sec:Weylalg} below),
\begin{equation*}
    (f\#_t g)(v)  = \frac{1}{(2\pi)^{2n}}\int_{H^*_p} e^{2i\,t\omega^*(x,y)}f(v+x)g(v+y) dx\,dy
\end{equation*} 
and in the $t=0$ hyperplane it is the pointwise product of functions.

Since a Heisenberg principal symbol $\sigma^m_H(T)$ is $\delta$-homogeneous, it is determined by its degree $m$ and its restriction  to the two hyperplanes with $t=1$ and $t=-1$.
We therefore identity the principal Heisenberg symbol $\sigma^m_H(T)$ with the resulting pair of functions,
\[ \sigma^m_H(T) = (\sigma^m_+(T),\sigma^m_-(T)) \in \mcW_H\oplus \mcW_H^{op}\]
Here $\sigma^m_{+/-}(T)$ are two smooth functions on $H^*$.
$\mcW_H$ is the bundle over $M$ whose fibers are the Weyl algebras $\mcW(H^*_p,\omega^*)$.
$\mcW_H^{op}$ is the bundle whose fibers are the opposite algebras $\mcW(H^*_p,-\omega^*)=\mcW(H^*_p,\omega^*)^{op}$.

The two functions $\sigma_+^m(T), \sigma_-^m(T)$ cannot be chosen randomly.
Assume we have chosen a compatible complex structure $J$ for $H$,
so that, in particular,  $H$ has a Euclidean structure.
Let $\rho_q$  be the function 
\[\rho_q(v,t):=\frac{t}{\|v\|^2}\qquad (v,t)\in H_p^*\times \RR,\; v\ne 0\]
Here $\|v\|$ is the Euclidean length of $v$.
The function $\rho_q$ is constant on parabolic rays in $H^*_p\times \RR$,
i.e. on subsets of the form $\{(sv,s^2t)\mid s>0\}$ with $v\ne 0$.
If we restrict $\rho_q$ to the unit sphere  $S(H^*_p\times \RR)$ we have $\|v\|^2+t^2=1$, and so $\rho_q=t/(1-t^2)$.
Thus, $\rho_q$ can be used as a smooth transverse coordinate near the equator of the unit sphere.
Near the equator $\sigma^m_H(T)$ has a Taylor expansion in powers of $\rho_q$.

On $H_p^*$ we let
\begin{equation*} \rho(v) := \frac{1}{\|v\|}\qquad v\in H^*_p,\; v\ne 0\end{equation*}
When we restrict $\rho_q$ to the hyperplane $H_p^*\times \{+1\}$ we have $\rho_q = \rho^2$,
while on  $H_p^*\times \{-1\}$ we have $\rho_q = -\rho^2$. 
The Taylor expansion of $\sigma^m_H(T)$ near the equator of the unit sphere in powers of $\rho_q$
corresponds to an asymptotic expansion of the functions $\sigma^m_+(T)$ and $\sigma^m_-(T)$ in powers of $\rho^2$ (for large $v$), of the form
\begin{equation*} \sigma(p,v)\sim \sum_{j=0}^\infty \sigma_{2j}(p)\rho(v)^{-m+2j}\qquad p\in M, v\in H_p^*\end{equation*}
and the  expansions of $\sigma_+=\sigma_+^m(T)$ and $\sigma_-=\sigma_-^m(T)$ are related by 
\begin{equation}\label{sigma_compatible}
(\sigma_-)_{2j} = (-1)^j(\sigma_+)_{2j}
\end{equation}
This condition guarantees that a pair of smooth functions $(\sigma_+,\sigma_-)$ on $H^*$ extends to a smooth function on $H^*\oplus \underline{\RR}$ (minus the zero section)  that is homogeneous of degree $m$ in each parabolic ray.
In this paper  we shall represent Heisenberg principal symbols  as such pairs of functions $(\sigma_+,\sigma_-)$.

\begin{remark}
Note that the graded algebra of principal Heisenberg symbols
\[ \Symball:=\bigoplus_{m\in \ZZ}\Symbm\qquad \Symbm:=\Psi^m_H/\Psi^{m-1}_H\]
is not a subalgebra of $\mcW_H\oplus \mcW_H^{op}$,
because $\msS^m\cap \msS^{k}$ is not zero if $m\equiv k\mod 4$.
Rather, we have an algebra homomorphism
\[ \Symball\to \mcW_H\oplus \mcW_H^{op}\]
which is injective in each degree $\Symbm$.
Also note that not all elements in $\mcW_H\oplus \mcW_H^{op}$ are Heisenberg principal symbols.
\end{remark}

\section{An enlarged symbol algebra}\label{sec:alg}

In this section we construct a $\ZZ$-graded algebra $\Alg_H$ that contains the algebra of order zero principal Heisenberg symbols $\Symb$ as a subalgebra.
The motivation for the introduction of $\Alg_H$ is that the curvature of $\Symb$, as a bundle of algebras on $M$,
can be represented as a commutator with a 2-form $\theta\in \Omega^2(\Alg_H)$.

\subsection{The Weyl calculus}
We briefly recall some basic facts about the Weyl calculus. (See \cites{GLS68, Ho79, Sh01}.)

A linear function $L(x,\xi) = \sum a_jx_j+b_j\xi_j$ in $2n$ real variables $x_1,\dots, x_n,\xi_1,\dots,\xi_n\in\RR$ is ``quantized'' as the differential operator
\begin{equation*} \Op(L) = \sum a_jX_j +b_j D_j\qquad D_j = -i\frac{\partial}{\partial x_j}\quad (j=1,\dots,n)\end{equation*}
where $X_j$ denotes multiplication by $x_j$. 
Weyl quantization is characterized by the fact that the exponential function $a(x,\xi) = e^{iL(x,\xi)}$ is quantized as $\Op^w(a) = e^{i\Op(L)}$.
This choice leads by superposition (i.e. Fourier decomposition) to a  formula for the quantization of a general function $a(x,\xi)$.
Formally, in the Weyl calculus the pseudodifferential operator $A=\Op^w(a)$ 
acts on functions $u\in \mcS(\RR^n)$ as 
\begin{equation}\label{eqn:Weyl}
 (Au)(x) = \frac{1}{(2\pi)^n} \iint e^{i(x-y)\cdot \xi} a\left(\frac{x+y}{2},\xi\right)u(y)\,dy\, d\xi
 \end{equation}
The Weyl symbol of the formal adjoint $A^t$ of $A$ is the complex conjugate of $a$,
\begin{equation*} \Op^w(a)^t = \Op^w(\bar{a})\end{equation*}
We shall be interested in the Weyl algebra $\mcW$ of smooth complex valued functions $a(x,\xi)\in C^\infty(\RR^{2n},\CC)$ 
for which   there is an integer $m\in \ZZ$ with,
\begin{equation}\label{seminorms1} 
|(\partial^\alpha_x\partial^\beta_\xi a)(x,\xi)| \le C_{\alpha,\beta} (1+\|x\|^2+\|\xi\|^2)^{(m-|\alpha|-|\beta|)/2}
\end{equation}
for every pair of multi-indices $\alpha, \beta$,
and such that $a(x,\xi)$ has a 1-step polyhomogeneous asymptotic expansion 
\begin{equation}\label{asymptotic expansion} 
a\sim \sum_{j=-m}^\infty a_j\qquad a_j(s x,s \xi)=s^{-j}a_j(x,\xi) \quad s>0, (x,\xi)\ne (0,0)\end{equation}
The integer $m\in \ZZ$ is the Weyl order of the operator $\Op^w(a)$. 
The notation  is as usual:
\begin{equation*}
\alpha = (\alpha_1,\dots, \alpha_{2n}),\quad \alpha_j=0,1,2,\dots\qquad 
|\alpha| = \alpha_1+\cdots + \alpha_{2n}\end{equation*}
\begin{equation*}\partial^\alpha_x = \left(\frac{\partial}{\partial x_1}\right)^{\alpha_1}
\cdots\left(\frac{\partial}{\partial x_n}\right)^{\alpha_n}\qquad 
\partial^\beta_{\xi} = \left(\frac{\partial}{\partial \xi_1}\right)^{\beta_1}
\cdots
\left(\frac{\partial}{\partial \xi_n}\right)^{\beta_{n}}
\end{equation*}
The formal series is an asymptotic expansion in the sense that the difference of $a(x,\xi)$ with a partial sum is bounded as follows.
Fix a cut-off function  $\varphi(x, \xi) \in C^\infty(\mathbb{R}^{2n})$ such that  $0 \le \varphi(x, \xi) \le 1$, $\varphi(x, \xi)=0$ when $\|x\|^2 +\| \xi\|^2 \le 1$, $\varphi(x, \xi)=1$ when $\|x\|^2 +\| \xi\|^2 \ge 2$. Then the expansion above is asymptotic in the sense that 
\begin{equation} \label{seminorms2} 
\left|\partial^\alpha_x\partial^\beta_\xi\left( a(x,\xi) - \varphi(x, \xi)\sum_{j=-m}^{N-1} a_j(x,\xi)\right)\right| \le C_{N, \alpha, \beta} (1+\|x\|^2+\|\xi\|^2)^{-N/2} 
\end{equation}
$a(x,\xi)$ is determined by its  asymptotic expansion modulo a Schwartz class function on $\RR^{2n}$.

We denote by $\mcW^m$ the space of symbols of order $m$,
\begin{equation*} \cdots\subset \mcW^{-2}\subset \mcW^{-1}\subset \mcW^{0}\subset \mcW^{1}\subset \mcW^{2}\subset \cdots\end{equation*}
For $a\in \mcW$, the operator $A=\Op^w(a)$ that is formally defined by  (\ref{eqn:Weyl}) is continuous as a linear map
\begin{equation*} A \colon \mcS(\RR^n)\to \mcS(\RR^n)\end{equation*}
Order zero operators are bounded on $L^2(\RR^n)$.
Operators of negative order are compact.
Operators of order less than $-2n$ are trace class.
The intersection of operators of all orders is the ideal of smoothing operators
with Schwartz kernels $k(x,y)\in \mcS(\RR^n\times \RR^n)$,
\begin{equation*} \mcS(\RR^{2n})=\bigcap_{m\in \ZZ} \mcW^m\end{equation*}
The product of two symbols $a\# b$ is defined indirectly by
\begin{equation*} \Op^w(a\#b) = \Op^w(a)\Op^w(b)\end{equation*}
An explicit formula for this  product  is
\begin{equation*}  (a \# b)(v) = \frac{1}{(2\pi)^{2n}}\iint e^{2i\omega(x,y)} a(v+x)b(v+y) \,dxdy\end{equation*}
where $\omega$ is the standard symplectic form on $\RR^{2n}$
\begin{equation*} \omega = \sum_{j=1}^n dx_j\wedge d\xi_j\end{equation*}
Note that $dx=dy=|\omega|^n/n!$. 

Asymptotically we have
\begin{equation}\label{eqn:asymptotic}
 a\#b \sim \sum_{k=0}^\infty \left(\frac{i}{2}\right)^k (a\# b)_k\qquad (a\# b)_k=\sum_{|\alpha|+|\beta|=k}\frac{1}{\alpha!\beta!}(-1)^{|\beta|}(\partial^\alpha_x\partial^\beta_\xi a)(\partial^\beta_x\partial^\alpha_\xi b)
 \end{equation}
The leading $k=0$ term is the pointwise product 
\begin{equation*} (a\#b)_0=ab\end{equation*}
and the $k=1$ term is the Poisson bracket
\begin{equation*}(a\#b)_1 = \sum_{j=1}^n \frac{\partial a}{\partial x_j}\frac{\partial b}{\partial \xi_j}-\frac{\partial a}{\partial \xi_j}\frac{\partial b}{\partial x_j}=:\{a,b\}\end{equation*}
Note that $ba=ab$ and $\{b,a\}=-\{a,b\}$. 
In general, 
\begin{equation*} (a\#b)_k = (-1)^k(b\#a)_k\end{equation*}
The asymptotic expansion of the product $\#$ is often called  {\em Moyal product}.
If $a, b$ are polynomials in $x,\xi$, the Moyal product is exactly equal to the product $a\#b$. 




\subsection{Moyal product}
Let $\mcB=\mcW/\mcS$ be the algebra of full symbols in the Weyl calculus,
with   quotient map $\lambda \colon \mcW\to\mcB$,
\begin{equation*} 0\to \Sch \to \mcW \stackrel{\lambda}{\longrightarrow} \mcB\to 0\end{equation*}
A Weyl symbol $a\in\mcW$ of degree $m\in \ZZ$ has an asymptotic expansion modulo symbols of smoothing operators in $\Sch(\RR^{2n})$,
as sums of homogeneous terms,
\begin{equation*}a(x,\xi)\sim \sum_{j=-m}^\infty a_j(x,\xi)\qquad a_j(sx,s\xi)=s^{-j}a_j(x,\xi)\end{equation*}
The expansion (\ref{eqn:asymptotic}) determines the product of formal Laurent series in $\mcB$.
If the formal series $a= \sum a_{l}$ and $b= \sum b_{m}$ represent  elements of $\mcB$,
then by (\ref{eqn:asymptotic}) the star product $a\star b=c$ is represented by the formal series
\begin{equation*}a \star b =  \sum  c_p\end{equation*}
where $c_p$ is homogeneous  of degree $-p$, and given by the finite sum
\begin{equation*}c_{p} := \sum_{2k+l+m=p}  B_k(a_{l},b_{m})
\end{equation*}
with
\begin{equation*} B_k(x, y)=\left(\frac{i}{2}\right)^k \sum_{|\alpha|+|\beta|=k}\frac{1}{\alpha!\beta!}(-1)^{|\beta|}(\partial^\alpha_x\partial^\beta_\xi x)(\partial^\beta_x\partial^\alpha_\xi y) \end{equation*}
Since $B_k(a_l,b_m)$ is homogeneous of degree $-l-m-2k$ if $a_l$, $b_m$ are homogeneous of degrees $-l$, $-m$ respectively, it follows that $c_{p}$ is indeed homogeneous of degree $-p$.

\subsection{The model algebra $\Alg$}
Let $\mcB_q\subset \mcB$ be the subset consisting of formal series 
$a=\sum_{l=-m}^\infty a_{2l}$ for which  all the terms $a_{2l}$ are homogeneous of even degree $-2l$.
It is clear from the explicit form of the Moyal product $\star$ that $\mcB_q$ is a subalgebra of $\mcB$.

\begin{definition}
Let $\iota \colon \mcB_q\to\mcB_q$ be the linear map 
\begin{equation*} a=\sum_{j=-m}^\infty a_{2j}\in \mcB_q \qquad \iota(a):= \sum_{j=-m}^\infty (-1)^ja_{2j}\in \mcB_q \end{equation*}
where $a_{2j}=a_{2j}(x,\xi)$ is homogeneous of degree $-2j$.
\end{definition}
\begin{lemma}\label{lem:i}
The map $\iota  \colon  \mcB_q\to\mcB_q$ is an involution,
\begin{equation*} \iota(a\star b) = \iota(b) \star \iota(a)\end{equation*}
\end{lemma}

\begin{proof}
With $a= \sum a_{2l}$, $b= \sum b_{2m}$ we have
\begin{multline*}
\iota(a \star b) = \\
\sum_p (-1)^p\sum_{2k+2l+2m=2p} B_k(a_{2l},b_{2m}) =
\sum_p \sum_{k+l+m=p} (-1)^k B_k((-1)^la_{2l},(-1)^mb_{2m}) =\\  
\sum_p \sum_{k+l+m=p} B_k((-1)^mb_{2m}, (-1)^la_{2l}) =\iota(b) \star \iota(a)
\end{multline*}
where we used that $B_k(x, y)=(-1)^kB_k(y, x)$.

\end{proof}

We let
\begin{equation*} \mcW_q:=\{w\in \mcW\mid \lambda(w)\in \mcB_q\}\qquad \mcW^{2m}_q:=\{w\in \mcW^{2m}\mid \lambda(w)\in \mcB_q\}\end{equation*}
and 
\begin{equation*} \Alg := \{(w_+,w_-)\in  \mcW_q\oplus \mcW_q^{op}\mid \lambda(w_+) =\iota \circ \lambda(w_-)\}\end{equation*}
Here $\mcW^{op}_q$ is the opposite algebra of $\mcW_q$, i.e. the product of elements in $\Alg$ is
\begin{equation*} (u_+,u_-)(w_+,w_-):= (u_+\#w_+,w_-\# u_-)\end{equation*}
That $\Alg$  is  an algebra follows from Lemma \ref{lem:i},
\begin{equation*} \iota \circ \lambda(w_-\#u_-)=\iota(\lambda(w_-)\star\lambda(u_-))=\iota(\lambda(u_-))\star\iota(\lambda(w_-))=\lambda(u_+)\star \lambda(w_+)=\lambda(u_+\#w_+)\end{equation*}
Note that $\Alg$ is a $\ZZ$-graded algebra. For $m\in \ZZ$, we let 
\begin{equation*} \Alg^{2m} := \{(w_+,w_-)\in  \mcW^{2m}_q\oplus \mcW^{2m}_q\mid \lambda(w_+) =\iota \circ \lambda(w_-)\}\end{equation*}

\subsection{The Weyl  algebra of a symplectic vector space}\label{sec:Weylalg}
As can be seen from the explicit formula for the $\#$ product, a linear symplectic  transformation $\phi \colon \RR^{2n}\to\RR^{2n}$  acts on the algebra 
$\mcW$ by automorphisms. For $\phi \in \Sp(\RR^{2n})$, $a\in \mcW$ let $\phi(a):= a \circ \phi^{-1}$.
Then
\begin{equation*} \phi(a)\# \phi(b) = \phi(a\#b).\end{equation*}
We may therefore define the Weyl algebra  for a general  finite dimensional symplectic vector space $V$ with symplectic form $\omega$.
The Weyl  algebra $\mcW(V,\omega)$ is the $\ZZ$-filtered algebra consisting of smooth functions  $f\in C^\infty(V)$ that have an asymptotic expansion
\begin{equation*} f \sim \sum_{j=-m}^\infty f_j  \qquad f_j(s v) = s^{-j}f(v)\quad v\ne 0,\;s>0\end{equation*}
The action of the symplectic group $\Sp(2n)$ preserves the ideal of smoothing operators $\mcS$ and therefore induces an action of $\Sp(2n)$ on the quotient $\mcB=\mcW/\mcS$ so that
\begin{equation*}
\phi(\lambda (f)) = \lambda (\phi(f)), \ f \in \mcW.
\end{equation*}
If we identify elements of $\mcB$ with the asymptotic expansions $\sum_{j=-m}^\infty f_j$ the action is given by $\phi(\sum_{j=-m}^\infty f_j)= \sum_{j=-m}^\infty f_j\circ \phi^{-1}$. It follows that the action of $\Sp(2n)$ preserves the subalgebra $\mcB^0$ of $\mcB$ as well as the subalgebra $\mcW_q$ of $\mcW$. The action of $\Sp(2n)$ on $\mcB^0$ is compatible with the anti-involution $\iota$:
\begin{equation*}
\phi (\iota (f)) = \iota (\phi(f)), \ f \in \mcB^0
\end{equation*}
As a consequence, for any symplectic vector space $(V, \omega)$ we obtain a well defined algebra $\mcB(V, \omega)$, homomorphism $\lambda  \colon  \mcW(V, \omega) \to \mcB(V, \omega)$, subalgebra $\mcB^0 (V, \omega) \subset \mcB (V, \omega)$, anti-involution $\iota$ of $\mcB^0(V, \omega)$, subalgebra $\mcW_q(V, \omega) \subset \mcW(V, \omega)$, and finally the algebra
\begin{equation*} \Alg(V,\omega) := \{(w_+,w_-)\in  \mcW_q(V,\omega)\oplus \mcW_q^{op}(V,\omega)\mid \lambda(w_+) =\iota \circ \lambda(w_-)\}\end{equation*}
The symplectic group $\Sp(V)$ acts on all these algebras and $\lambda$, $\iota$ are equivariant with respect to this action.

\subsection{The algebra $\Alg_H$ for contact manifolds}
On a contact manifold $M$ with contact form $\alpha$, the fibers of the bundle $H=\Ker \alpha$ are symplectic vector spaces $(H_p,\omega_p)$, with $\omega:=-d\alpha$.
We denote by $\mcW_H$ the algebra of smooth sections in the bundle over $M$
whose fiber at $p\in M$ is $\mcW(H_p^*,\omega^*_p)$.
As in section \ref{sec:symbols}, the principal Heisenberg symbol of an operator $\msL$ of order $m$ is an element 
\[ \sigma^m_H(\msL)=(\sigma_+,\sigma_-)\in \mcW_H\oplus \mcW_H^{op}\]
We denote by $\Alg_H$ the algebra of smooth sections in the bundle over $M$
whose fiber at $p\in M$ is $\Alg(H_p^*,\omega^*_p)$.
Principal Heisenberg symbols of order zero are a subalgebra of $\Alg_H$,
\[ \Symb\subset \Alg_H\subset \mcW_H\oplus \mcW_H^{op}\]
For $m\in \ZZ$, we likewise define $\Alg^m_H\subset \Alg_H$ with fiber $\Alg^m(H_p^*,-\omega_p)$.
Then,
\[ \Symb=\Alg^0_H\]

\section{A curved dga}\label{section_curvature}

For a contact manifold $M$, a symplectic connection on the bundle $H\subset TM$
determines a connection $\con \colon \Alg_H\to \Omega^1(\Alg_H)$.
In this section we construct a 2-form $\curv\in \Omega^2(\Alg_H)$ with  
$\con^2(a)=[\curv,a]$ and $\con(\curv)=0$.
The triple $(\Omega^\bullet(\Alg_H), \con, \curv)$ is a curved dga.

\subsection{The action of the symplectic Lie algebra on $\Alg$}


Let $[a,b]:= a\#b-b\#a$ denote the commutator of elements in the Weyl algebra $\mcW(V,\omega)$
of a symplectic vector space $(V,\omega)$.
The asymptotic expansion (\ref{eqn:asymptotic}) of $a\# b$ is an equality if $a$ is a  polynomial.
In particular, if $a$ is a homogeneous polynomial of degree 2 then $[a,b] = i\{a,b\}$.
If, moreover, $b$ is a homogeneous function on $V$ then 
$i\{a,b\}$ is homogeneous of the same degree as $b$.

Let $\mathfrak{g}\subset \mcW$ be the subspace of  homogeneous polynomials of degree $2$ that are purely imaginary (i.e. with values in $i\RR$).
If  $X, Y \in \mathfrak{g}$ then   $[X,Y] =i\{X, Y\}$ is purely imaginary and homogeneous of degree 2.
Thus, $[X,Y]\in \mathfrak{g}$, and so $\mathfrak{g}$ is a (real) Lie algebra.


Let $V^*:=\Hom(V, \mathbb{R})\subset \mcW$ be the (real) dual space of $V$.
If   $X\in \mathfrak{g}$ and $f \in V^*$ then  $[X, f] =\{iX,f\}$ is real-valued and homogeneous of degree 1, and so  $[X,f]\in V^*$.
We obtain a morphism of Lie algebras,
\begin{equation*}
\mu^*  \colon  \mathfrak{g} \to \End V^* \qquad \mu^*(X):= [X,\,\cdot\,]
\end{equation*}
The map $V\ni v \mapsto \omega(v, \cdot) \in V^*$ establishes an isomorphism $V \to V^*$ which can be used to define a symplectic form $\omega^*$  on $V^*$.
For $f$, $g \in V^*$ we have 
$ \{f, g\}=\omega^*(f, g)\cdot 1$,
where  $1\in \mcW$ is a constant function on $V$.
 
\begin{lemma}
If  $X \in \mathfrak{g}$ then $\mu^*(X) \in \mathfrak{sp}(V^*)$.
The map $ \mu^*  \colon  \mathfrak{g} \to \mathfrak{sp}(V^*)$ is a Lie algebra isomorphism.
\end{lemma}
\begin{proof}

If $X\in \mathfrak{g}$ and $f$, $g \in V^*$ then
\begin{equation*}(\omega^*(\iso(X)f, g)+ \omega^*(f, \iso(X)g))\cdot 1 =  i\{\{X,f\}, g\}+i\{f, \{X,g\}\}=
i\{X, \{f, g\}\}=0
\end{equation*}
This proves that $\mu^*(X)$ is in $\mathfrak{sp}(V^*)$.

The Lie algebra morphism $\mu^*$ is injective. 
$\mu^*(X) = \{iX,\,\cdot \,\}$ acts by the Hamiltonian vector field of $iX$,
which is zero only if $X$ is constant, i.e. $X=0$.
Surjectivity follows from the observation that $\dim \mathfrak{g}= n(2n+1) =\dim \mathfrak{sp}(V^*)$, where $2n =\dim V$.

\end{proof}
We will identify the groups of symplectic transformations of $\mathrm{Sp}(V)\cong \mathrm{Sp}(V^*)$ as well as the corresponding Lie algebras $\mathfrak{sp}(V)\cong \mathfrak{sp}(V^*)$ via the isomorphism $V\to V^*$: $v \mapsto \omega(v, \cdot)$.
We obtain a corresponding isomorphism of Lie algebras
\[ \iso \colon \mathfrak{g}\to \mathfrak{sp}(V)\]
As in  section \ref{sec:Weylalg}, the symplectic group $\mathrm{Sp}(V)$ acts  on $\mcW(V,\omega)$ by algebra automorphisms.
Therefore the Lie algebra $\mathfrak{sp}(V)$ acts on $\mcW(V,\omega)$ by derivations,
\[ (\phi.w)(v):=-w(\phi(v))\qquad v\in V,\, w\in \mcW(V,\omega),\;\phi \in \mathfrak{sp}(V)\]


\begin{lemma}\label{act}
The action of the Lie algebra $\mathfrak{sp}(V)$  on $\mcW(V,\omega)$ is by inner derivations,
\[ \phi.w = [\mu^{-1}(\phi),w]\qquad \phi\in \mathfrak{sp}(V),\;w\in \mcW\] 
\end{lemma}
\begin{proof}
The two actions $w\mapsto \phi.w$ and $w\mapsto [\mu^{-1}(\phi),w]$ of $\mathfrak{sp}(V)$ on $\mcW(V,\omega)\subset C^\infty(V)$ are given by differentiation along vector fields on $V$. Therefore it is sufficient to verify that the corresponding vector fields coincide. To verify that two vector fields coincide it is sufficient to compare their actions on linear functions on $V$, i.e. their actions on $V^* \subset \mcW(V,\omega)$. 

Note that the action of $\mathfrak{sp}(V)$ on $V^*$ that is dual to its canonical action on $V$ is equal to the action of $\mathfrak{sp}(V)$ on $V^*$ determined by the isomorphism $V\to V^*$, $v\mapsto \omega(v,\cdot)$.
Therefore, if $X=\mu^{-1}(\phi)$ then the action of $\phi=\mu(X)\in \mathfrak{sp}(V)$ on $V^*$
is equal to the action of $\mu^*(X)\in \mathfrak{sp}(V^*)$ on $V^*$,
\[ \phi.f = \mu(X)f=\mu^*(X)f=[X,f]\qquad f\in V^*\]

\end{proof}

The symplectic group $\mathrm{Sp}(V)$  acts on $\Alg(V,\omega)$ by automorphisms,
and the Lie algebra $\mathfrak{sp}(V)$ acts on $\Alg(V,\omega)$ by derivations.
Lemma \ref{act} implies that with $\phi\in \mathfrak{sp}(V)$ and $(w_+,w_-)\in \Alg(V,\omega)$,
\begin{multline*}
    \phi.(w_+,w_-) = (\phi.w_+,\phi.w_-)=([\mu^{-1}(\phi),w_+],[w_-,\mu^{-1}(\phi)])=\\
    ([\mu^{-1}(\phi),w_+],[-\mu^{-1}(\phi),w_-])=
    [(\mu^{-1}(\phi),-\mu^{-1}(\phi)), (w_+,w_-)]
\end{multline*} 
Note that since $\mu^{-1}(\phi)\in \mathfrak{g}$ is homogeneous of order 2,
we have $(\mu^{-1}(\phi),-\mu^{-1}(\phi))\in \Alg$.
Let $\nu$ be the map
\begin{equation}\label{Iso}
\nu \colon \mathfrak{sp}(V) \to \Alg \qquad \Iso(\phi):= (\iso^{-1}(\phi), -\iso^{-1}(\phi)) 
\end{equation}
The map $\nu$ is a morphism of Lie algebras.
\begin{proposition}\label{action_is_inner}
The action of the Lie algebra  $\mathfrak{sp}(V)$ on $\Alg(V,\omega)$ is by inner derivations, 
\begin{equation*}
\phi.(w_+,w_-) = [\Iso(\phi), (w_+,w_-)]\qquad \phi \in \mathfrak{sp}(V),\;(w_+,w_-)\in \Alg(V,\omega)
\end{equation*}
\end{proposition}

\subsection{The curvature of $\Alg_H$}\label{symplectic bundle}
Let $M$ be a smooth  manifold of dimension $2n+1$ with a contact  1-form $\alpha$.
We use the notation established in section \ref{sec:contact}.
In particular, recall that $H=\Ker \alpha\subset TM$ is a symplectic bundle with symplectic form $\omega=-d\alpha$, and that $P_H$ is the principal $\mathrm{Sp}(2n)$-bundle of symplectic frames of $H$.

The symplectic group $\mathrm{Sp}(2n)$ acts on the algebra $\Alg$ by automorphisms.
Let $\mathcal{A}_H$ be the associated bundle
\[\mathcal{A}_H:= P_H\times_{Sp(2n)}\Alg\]
We identify $\Alg_H$ with the algebra of smooth sections of the bundle $\mathcal{A}_H$.
Smooth sections in $\mathcal{A}_H$
are  smooth $\mathrm{Sp}(2n)$-invariant functions from $P_H$ to $\Alg$,
\begin{equation*}
\Alg_H = C^\infty(M;\mathcal{A}_H) =C^\infty(P_H; \Alg)^{Sp(2n)}
\end{equation*}
where the action of $\phi \in \mathrm{Sp}(2n)$ on $s  \colon P_H\to  \Alg$ is given by $(\phi \cdot s)(p):= \phi(s(p\phi))$.
In other words, $s$ is invariant if $s(p\phi)=\phi^{-1}(s(p))$.

Let  $\Omega^\bullet(P_H; \Alg)_{basic}$ be the space of basic $\Alg$-valued forms. 
Recall that an $\Alg$-valued  differential form  $\eta$ on $P_H$ 
is called basic if:
\begin{itemize}
    \item $\eta$ is horizontal, i.e. $\iota_X \eta=0$ for every vertical vector field $X$ on $P_H$;
    \item $\eta$ is $\mathrm{Sp}(2n)$ invariant, i.e. $\phi^* \eta = \phi^{-1}(\eta)$ for every $\phi \in \mathrm{Sp}(2n)$. 
\end{itemize}
$k$-forms with values in the bundle $\mathcal{A}_H$ are, by definition, basic $\Alg$-valued $k$-forms.
We denote
\begin{equation*}
 \Omega^k(\Alg_H) = \Omega^k(M;\mathcal{A}_H) = \Omega^k(P_H; \mathcal{A}_H)_{basic}
\end{equation*}
A symplectic connection $\nabla$ on $H$ can be represented by a  connection $1$-form 
\[ \beta \in  \Omega^1(P_H; \mathfrak{sp}(2n))_{basic}\cong \Omega^1(M, \mathfrak{sp}(H))\]
The curvature of $\nabla$ is
\[ \theta := d\beta +\frac{1}{2}[\beta, \beta] \in \Omega^2(P_H, \mathfrak{sp}(2n))_{basic}\cong \Omega^2(M, \mathfrak{sp}(H))\]
The 1-form $\beta$ defines a  covariant derivative 
\begin{equation*}
\con  \colon  \Alg_H \to  \Omega^1(\Alg_H)
\end{equation*}
by
\begin{equation*}
\con (a) :=da +\beta \cdot a\in \Omega^1(\Alg_H) \qquad a \in  \Alg_H = C^\infty(P_H; \Alg)^{Sp(2n)}
\end{equation*}
This covariant derivative extends to a derivation
\begin{equation*}
\con  \colon  \Omega^k(\Alg_H) \to \Omega^{k+1}(\Alg_H)
\end{equation*}
by the same formula,
\begin{equation*}
\con (\eta) :=d\eta +\beta \cdot \eta  
\end{equation*}
The curvature of $\con$ is
\begin{equation*}
\con^2 (\eta)= \theta\cdot \eta
\end{equation*}
By Proposition \ref{action_is_inner},
\begin{equation}\label{defcon}
\con (\eta) = d\eta +[\Iso(\beta), \eta] \qquad \con^2 (\eta) = [\nu(\theta), \eta] 
\end{equation}
Define
\begin{equation}\label{deftheta}
\curv:= \Iso (\theta) \in  \Omega^2(\Alg_H)  
\end{equation}
so that
\begin{equation*}
    \con^2(\eta) = [\curv,\eta]\qquad  \eta\in \Omega^\bullet(\Alg_H)
\end{equation*}

\begin{lemma} With the definitions above we have
\begin{equation*}
    \con(\curv)=0\\
\end{equation*}
\end{lemma}
\begin{proof}
Since $\nu \colon \mathfrak{sp}(2n)\to \Alg$ is a Lie algebra morphism, we have 
\begin{equation*}
\con (\Iso(\eta)) =d \Iso(\eta) +[\Iso(\alpha), \Iso(\eta)] =\Iso(d \eta+[\alpha, \eta])=\Iso(\nabla(\eta)).
\end{equation*}
for any $\eta \in \Omega^k(P_H, \mathfrak{sp}(2n))_{basic}$.
Hence
\begin{equation*}\con(\curv) = \con (\nu (\theta)) = \nu (\nabla(\theta)) =0, \end{equation*}
since by the Bianchi identity $\nabla(\theta)=0$.

\end{proof}

Finally, let us record the dependence of $\con$ and $\curv$ on the choice of symplectic connection.
\begin{lemma}\label{changeofconcurv}
Let $\nabla' = \nabla +\kappa$ be another symplectic connection on $H$ with $\kappa \in \Omega^1(M, \mathfrak{sp}(H))$, and let $\con'$, $\curv'$ be as in  \eqref{defcon} and \eqref{deftheta}. Then,
\begin{align}
&\con'= \con + [\boldsymbol{\kappa}, \,\cdot\,]\\
&\curv'=\curv+\con(\boldsymbol{\kappa})+\boldsymbol{\kappa}^2
\end{align}
where $\boldsymbol{\kappa}:= \Iso(\kappa)\in \Omega^1(\Alg_H)$.
\end{lemma}
\begin{proof}
The first identity follows immediately from the equation \eqref{defcon}. For the second, notice that the curvature $\theta'$ of $\nabla'$ is given by $$\theta'=\theta+ \nabla(\kappa)+\frac{1}{2}[\kappa, \kappa]$$
Hence $$\curv' = \curv+ \nu(\nabla(\kappa))+\frac{1}{2}[\nu(\kappa), \nu(\kappa)]=  \curv+\con(\boldsymbol{\kappa})+\boldsymbol{\kappa}^2$$
\end{proof}

\section{The trace}\label{section_trace}

If $\sigma_+,\sigma_-\in \Sch(H^*)$ are Schwartz class functions, then let $\tau$ be the combined trace
\[ \tau(\sigma_+,\sigma_-):=\Tr(\sigma_+)+(-1)^{n+1}\Tr(\sigma_-)\]
In this section we show how $\tau$ extends to a trace $\tau \colon \Alg\to \CC$,
and determines a graded trace
\[ \tau \colon \Omega^\bullet(\Alg_H)\to \Omega^\bullet(M)\]
with the property $\tau(\nabla(a))=d\tau(a)$.
We also prove the important identity
\[\tau(\curv^k)=0\qquad k=0,1,2,\dots\]

\subsection{Complex orders in the Weyl calculus}

The definition of Weyl symbols and Weyl pseudodifferential operators can be extended to include  complex orders. 
For $z \in \mathbb{C}$, $\mcW^z$ consists of smooth complex valued functions $a(x,\xi)\in C^\infty(\RR^{2n},\CC)$ 
such that 
\begin{equation*} |(\partial^\alpha_x\partial^\beta_\xi a)(x,\xi)| \le C_{\alpha,\beta} (1+\|x\|^2+\|\xi\|^2)^{(\re z-|\alpha|-|\beta|)/2}\end{equation*}
for every pair of multi-indices $\alpha, \beta$,
and where  $a(x,\xi)$ admits an  asymptotic expansion 
\begin{equation*} a\sim \rho^{-z}\sum_{j=0}^\infty a_j,\qquad a_j(s x,s \xi)=s^{-j}a_j(x,\xi) \quad s>0, (x,\xi)\ne (0,0)\end{equation*}
Here $\rho^{-z} = e^{-z \log \rho}$ with real valued $\log \rho \in \mathbb{R}$. 
For $z_1, z_2 \in \mathbb{C}$ we have
\begin{equation*}\mcW^{z_1} \# \mcW^{z_2} \subset \mcW^{z_1+z_2}\end{equation*}
We are mainly interested in the subset $\mcW_q^z \subset \mcW^z$ consisting of symbols $b\in \mcW^z$ which admit a step-2 polyhomogeneous expansion:
\begin{equation*} b\sim \rho^{-z} \sum_{j=0}^\infty b_{2j},\qquad b_{2j}(s x,s \xi)=s^{-2j}b_{2j}(x,\xi)\end{equation*}
It is immediate from the composition formula that
\begin{equation*}
\mcW_q^{z_1} \# \mcW_q^{z_2} \subset \mcW_q^{z_1+z_2}\qquad 
\end{equation*}


\subsection{The harmonic oscillator and its complex powers}
We denote by $\ho$ the harmonic oscillator
\begin{equation*}
\ho= \sum \limits_{j=1}^n \left(-\frac{\partial^2}{\partial x_j^2}+x_j^2\right)= \Op^w\left(\sum \limits_{j=1}^n (\xi_j^2+x_j^2)\right).
\end{equation*}
$\ho$ is a strictly positive selfadjoint operator.
The spectral theorem allows one to define for $t>0$ an operator $e^{-t\ho}$. It follows from Mehler's formula that
\begin{equation*}e^{-t\ho} = \Op^w(h_t)\qquad h_t\in \mcS(\RR^{2n})\end{equation*}
where 
\begin{equation}\label{Mehler}
h_t(x,\xi)=\frac{1}{(\cosh t)^n}e^{-(\|x\|^2 +\|\xi\|^2)\tanh t}    
\end{equation}
Using the spectral theorem one can also define complex powers $\ho^{-z}$ for $z \in \mathbb{C}$ . 
It is well known  (see e.g. \cites{guil}) that $\ho^{-z}$  is a Weyl pseudodifferential operator of order $-2z$,
\begin{equation*}\ho^{-z}= \Op^w(h^{-z})\qquad h^{-z}\in \mcW^{-2z}\end{equation*}
and $h^{-z}=h^{-z}(x, \xi)$ is an entire function of $z$.
This means, more precisely, that $z \mapsto (1+\|x\|^2+\|\xi\|^2)^zh^{-z}$ is a holomorphic function on $\mathbb{C}$ with values in the Frechet space $\mcW^0$.
Let 
\begin{equation*}h^{-z} \sim \rho^{2z}\sum_{j=0}^\infty  h_{j}(z)
\end{equation*}
be the asymptotic expansion of $h^{-z}$, where $h_{j}(z) = h_{j}(z, x, \xi)$ is homogeneous of degree $-j$ in $(x,\xi)\ne (0,0)$.

\begin{proposition} \label{entireh} 
The asymptotic expansion of  $h^{-z}$ is
 4-step homogeneous. More precisely, there exist entire functions $b_k(z)$, $k=0,1,2,\dots$ such that 
\begin{equation*}h_{j}(z) = \begin{cases}   b_{j/4}(z) \rho^{j} &\text{ if $j$ is divisible by $4$}\\
0 &\text{ otherwise}
\end{cases} 
\end{equation*}

\end{proposition} 
\begin{proof}
Information about the homogeneous terms $h_{j}$ can be obtained from Mehler's formula as follows. 
For $\re z >0$
\begin{equation*}
\Gamma(z)  \ho^{-z}= \int_0^\infty t^{z-1}  e^{-t\ho}dt
\end{equation*}
and so
\begin{equation*}
\Gamma(z) h^{-z}(x, \xi) =  \int_0^\infty   \frac{t^{z-1}}{(\cosh t)^n}e^{-(\|x\|^2 +\|\xi\|^2)\tanh t}dt
\end{equation*}
Fix any $\varepsilon >0$ sufficiently small, as specified precisely below. Note that 
\begin{equation*}\int_\varepsilon^\infty   \frac{t^{z-1}}{(\cosh t)^n}e^{-(\|x\|^2 +\|\xi\|^2)\tanh t}dt\end{equation*}
is an entire function of $z$ with values in $\mcS(\mathbb{R}^{2n})$. 
Therefore, to determine the asymptotic expansion of $h^{-z}$ in powers of $\rho$ we must analyze 
\begin{equation*}\int_0^\varepsilon  \frac{t^{z-1}}{(\cosh t)^n}e^{-(\|x\|^2 +\|\xi\|^2)\tanh t}dt =\int_0^\varepsilon  \frac{t^{z-1}}{(\cosh t)^n}e^{-\frac{\tanh t}{\rho^2}}dt\end{equation*}
By the change of variables $u = \tanh t$ we rewrite this integral as
\begin{equation*}
 \int_0^\epsilon  u^{z-1}\phi_z(u)e^{-\frac{u}{\rho^2}} du\end{equation*}
 where $\epsilon = \tanh(\varepsilon)$ and
 \begin{equation*}\phi_z(u)= \left(\frac{\tanh^{-1}(u)}{u}\right)^{z-1}(1-u^2)^{(n-2)/2}
\end{equation*}
The function $\phi_z(u)$ is  analytic near $u=0$, while $\phi_z(0)=1$ and $\phi_z(-u) = \phi_z(u)$.
Hence 
\begin{equation*}
\phi_z(u) = \sum_{k=0}^\infty a_k(z)u^{2k}\quad a_0(z)=1
\end{equation*}
where $a_k(z)$ are  entire functions of $z$.
Chose $\varepsilon>0$ so that $\phi_z(u)$ is analytic for  $|u| \le \epsilon =\tanh(\varepsilon)$. 
Since for every $k$ we have \begin{equation*}\int_\epsilon^\infty  u^{z-1}u^{2k}e^{-\frac{u}{\rho^2}} du =\mcO(\rho^\infty)\qquad  \text{as}\; \rho \to 0\end{equation*}
it follows that
\begin{multline*}
    \int_0^\epsilon  u^{z-1} \left( \sum_{k=0}^N a_k(z)u^{2k} \right)e^{-\frac{u}{\rho^2}} du = \int_0^\infty  u^{z-1} \left( \sum_{k=0}^N a_k(z)u^{2k} \right)e^{-\frac{u}{\rho^2}} du +\mcO(\rho^\infty) = \\
    \rho^{2z}\left(\sum_{k=0}^N a_k(z) \Gamma(z+2k)\rho^{4k}\right) +\mcO(\rho^\infty)
\end{multline*}
$\mcO(\rho^\infty)$ means $\mcO(\rho^m)$ for all $m>0$.
Now $\left| \phi_z(u)  -\sum_{k=0}^N a_k(z)u^{2k}\right| =\mcO(u^{2k+2}) $, and hence
\begin{equation*} \left|\int_0^\epsilon  u^{z-1} \left( \phi_z(u)-\sum_{k=0}^N a_k(z)u^{2k}\right)e^{-\frac{u}{\rho^2}} du \right|=\mcO(\rho^{2 \re z +4k+4}).\end{equation*}
We conclude that for $\re z >0$
\begin{equation*}
h^{-z} \sim \frac{1}{\Gamma(z)}\sum_{k=0}^\infty a_k(z) \Gamma(z+2k)\rho^{2z+4k}
\end{equation*}
and hence (for $\re z >0$)
\begin{equation*}h_{j}(z) = \begin{cases} \left(\prod \limits_{k=0}^{j-1} (z+k)\right) a_{j/4}(z) \rho^{j} &\text{ if $j$ is divisible by $4$}\\
0 &\text{ otherwise}.
\end{cases} 
\end{equation*}
Since both sides are entire functions of $z$ the equality holds for any $z \in \mathbb{C}$.

\end{proof}


As an application of Proposition \ref{entireh} we obtain the following result.

\begin{proposition}
\label{automorphism}
Let $w_+, w_- \in \mcW_q$ be such that $\iota \circ \lambda (w_+) = \lambda(w_-)$. Then
\begin{equation*}
\iota \circ \lambda (h^{z} \# w_+ \#  h^{-z}) = \lambda (h^{-z} \# w_- \#  h^{z}) \end{equation*}
In other words, if $(w_+,w_-)\in \Alg$ then $(h^{z} \# w_+ \#  h^{-z}, h^{-z} \# w_- \#  h^{z})\in \Alg$.
\end{proposition}
\begin{proof}
For $z\in \CC$ denote
$\mcB_q^z:=\mcW_q^z/\mcS$, and as before let $\lambda  \colon  \mcW_q^z\to \mcB_q^z$ be the quotient map. 
Let 
\begin{equation*}\mcB(z):=  \rho^{-z}\mcB_q = \bigcup_{m\in \ZZ} \mcB_q^{z+2m}\end{equation*}
For each $z\in \CC$ we  define $\iota_z$ as
\begin{equation*} \iota_z  \colon  \mcB(z)\to \mcB(z)\qquad \iota_z(\rho^{-z}b) := \rho^{-z}\iota(b)\qquad b\in \mcB_q\end{equation*}
If $z=0$ then $\mcB(0)=\mcB_q$, and $\iota_0$ is equal to the previously defined $\iota \colon \mcB_q\to \mcB_q$.
Note that  $\mcB(z)=\mcB(z+2)$, but that
\begin{equation*}\iota_{z+2}=-\iota_z\end{equation*}
If $a \in \mcB(z_1)$ and $b\in\mcB(z_2)$ then  $a\star b\in \mcB(z_1+z_2)$, and
\begin{equation*}
\iota_{z_1+z_2} (a\star b) = \iota_{z_2}(b) \star \iota_{z_1}(a)
\end{equation*}
The proof is essentially the same as that of Lemma \ref{lem:i}.

Proposition \ref{entireh} implies that
\begin{equation*}\iota_{2z} \circ \lambda(h^{z}) = \lambda(h^{z})  \end{equation*}
Then
\begin{multline*}
\iota \circ \lambda (h^{z} \# w_+ \#  h^{-z}) = \iota \left( \lambda (h^{z}) \star \lambda(w_+) \star  \lambda(h^{-z}) \right)= \iota_{-2z} \circ \lambda (h^{-z}) \star \iota \circ \lambda(w_+) \star  \iota_{2z} \circ \lambda (h^{z}) =\\  \lambda (h^{-z}) \star  \lambda(w_-) \star   \lambda (h^{z})= \lambda (h^{-z} \# w_- \#  h^{z}).
\end{multline*}
\end{proof}

\begin{remark}
The map $(w_+,w_-)\mapsto (h^{z} \# w_+ \#  h^{-z}, h^{-z} \# w_- \#  h^{z})$ is an automorphism of $\Alg$ that is formally like conjugation with $(h^{z},h^{z})$. But unless $z$ is an integer divisible by $4$, $(h^{z},h^{z})$ is  not an element in $\Alg$, nor is it an element in a larger algebra that contains $\Alg$.
\end{remark}

\subsection{A regularized trace for the Weyl algebra}
If $a\in \mcS(\RR^{2n})$, then according to equation (\ref{eqn:Weyl}) the Schwartz kernel of $\Op^w(a)$ is
\begin{equation*} K(x,y) = \frac{1}{(2\pi)^n} \int e^{i(x-y)\cdot \xi} a\left(\frac{x+y}{2},\xi\right)d\xi\end{equation*}
and so $\Op^w(a)$ is  a trace class operator, with trace
\begin{equation*} \Tr(\Op^w(a)) =\frac{1}{(2\pi)^n} \int a(x,\xi) dx \, d\xi \qquad a\in \Sch(\RR^{2n})\end{equation*}
This equation defines a trace on the algebra $\mcS$ (with product $\#$) which we  denote $\Tr(a)$.
If $a\in \mcW^z$  is of complex order $z$ with $\re\,z<-2n$, then $\Op^w(a)$ is trace class with the trace given by the same formula.

For $a\in \mcW_q^{2m}$ of even order $2m$, define the zeta-function
\begin{equation*} \zeta_a(z) := \Tr(\Op^w(a)\ho^{-z})\qquad \qquad  \re z> n+m\end{equation*}
Note that $\Op^w(a)\ho^{-z}$ is of order $-2m-2z$, and hence trace class if $\re z> n+m$.
It follows that
the zeta function is holomorphic for $\re z> n+m$. It
extends to a meromorphic function with at most simple poles at $m+n, m+n-1, m+n-2, \dots$.
The residue at $z=0$  of the  zeta-function gives a residue trace on $\mcW_q$,
\begin{equation*} \Res \colon \mcW_q\to \CC\qquad \Res(a) = \lim_{z\to 0} z\zeta_a(z) \end{equation*}
$\Res$ is a trace on $\mcW_q$ that vanishes on the ideal $\Sch$.
It follows that residue induces a trace on the quotient $\mcB_q=\mcW_q/\Sch$ which we also denote $\Res$.
An explicit formula for  $\Res a$ in terms of the asymptotic expansion $a=\sum_ja_{2j}\rho^{2j}\in \mcB_q$ is 
\begin{equation*}
\Res a= -\frac{1}{2(2\pi)^n}\int_{S^{2n-1}}a_{2n}(\theta) d \theta.
\end{equation*}
From this the following is immediate:
\begin{proposition}\label{antisymmetry}
For $a \in \mcB_q$
\begin{equation*}
\Res (\iota a)=(-1)^n \Res a
\end{equation*}
\end{proposition}
We denote by $\Trh (a)$ the constant term at $z=0$ of the zeta-function,
\begin{equation*}  \Trh(a) = \lim_{z\to 0}  \left(\zeta_a(z) - \frac{1}{z}\Res(a)\right)\end{equation*}
If $a\in \mcS$ then $\Tr(\Op^w(a)\ho^{-z})$ is an entire function, and we see that
\begin{equation*} \Trh(a) = \Tr(a)\qquad \forall a\in \Sch(\mathbb{R}^{2n})\end{equation*}
However the functional $\Trh$ is not a trace on $\mcW_q$.
\begin{example}\label{example 2-dim trace}
Let $n=1$, and take $a=(x^2+\xi^2)^{\#k}$ with $\Op^w(a)=\ho^k$, $k=0,1,2,3, \dots$. Then 
\begin{equation*}
\zeta_a(z)= \Tr \ho^{k-z}
\end{equation*}
The spectrum of $\ho$ consists of simple eigenvalues $1,3,5,7, \ldots$. 
Hence 
\begin{equation*}\Tr \ho^{k-z} = \sum_{l=0}^\infty (1+2l)^{k-z} =(1-2^{k-z})\zeta(z-k) \end{equation*}
where $\zeta(z)$ is the Riemann $\zeta$-function.
The equality holds if $\re z > k+1$, and hence we obtain equality 
\begin{equation*}\zeta_a(z) = (1-2^{k-z})\zeta(z-k)\end{equation*}
of the meromorphic extensions of the two functions.
In particular,
\begin{equation*}\Trh((x^2+ \xi^2)^{\#k})=(1-2^k)\zeta(-k)=\begin{cases} 0 &\text{if $k \ge 0$ is even}\\
                                             (2^k-1)\frac{B_{k+1}}{k+1}&\text{if $k >0$ is odd}
                                \end{cases}
\end{equation*} 
where $B_{k+1}$ is the Bernoulli number.
\end{example}

\subsection{The trace $\tau$ on $\Alg$}

\begin{lemma}\label{lem:pole}
If $\sigma=(w_+,w_-)\in \Alg$, then the meromorphic function
\begin{equation*} \varphi(z,\sigma) := \zeta_{w_+}(z)-(-1)^n\zeta_{w_-}(z)\end{equation*}
is holomorphic at $z=0$.
\end{lemma}
\begin{proof}
The zeta functions have at most simple poles at $z=0$.
The residue of $\varphi(z,\sigma)$ at $z=0$ is
\begin{equation*} \Res(w_+)-(-1)^n\Res(w_-)=\Res(\lambda(w_+))-(-1)^n\Res(\lambda(w_-))\end{equation*}
Since $\iota \circ\lambda(w_+)=\lambda(w_-)$, this residue is zero by Proposition \ref{antisymmetry}.

\end{proof}

\begin{definition}
Define a linear map $\tau  \colon \Alg\to \CC $ by 
\begin{equation*}  \tau(\sigma) := \varphi(0,\sigma)= \Trh(w_+)- (-1)^n \Trh(w_-)\end{equation*}
for $\sigma=(w_+,w_-)\in \mcA$.
\end{definition}

\begin{theorem}\label{tau is a trace}
$\tau$ is a trace, i.e. $\tau(\sigma\sigma')=\tau(\sigma'\sigma)$ for all $\sigma,\sigma'\in\Alg$.
\end{theorem}
\begin{proof}
For  $z \in \mathbb{C}$ and $\sigma = (w_+, w_-) \in \Alg$ define
\begin{equation*}\delta_z(\sigma):=(h^{-z}\# w_+\# h^{z},  h^{z} \# w_-\# h^{-z})\end{equation*} 
By Proposition \ref{automorphism} we have $\delta_z(\sigma)\in \Alg$. 
Since $h^{z_1}\# h^{z_2} =h^{z_1+z_2}$,  $\delta_z$ defines a one-parameter group $\delta \colon \CC\to \mathrm{Aut}(\Alg)$ of automorphisms of $\Alg$.
The function $z \mapsto \delta_z(\sigma)$ is entire for any $\sigma\in \Alg$. 
We  obtain a corresponding derivation $D  \colon  \Alg \to \Alg$ defined by
\begin{equation*}
D(\sigma) := \left. \frac{d}{dz}\right|_{z=0} \delta_z(\sigma)
\end{equation*}
The Taylor series of the function $\CC\to \Alg\; \colon \;z\mapsto \delta_z(\sigma)$ is
\begin{equation}\label{taylor}
\delta_z(\sigma) = \sum_{k=0}^\infty \frac{z^k}{k!}D^k(\sigma)
\end{equation}
Note that for $\sigma \in \Alg^m$ we have $D(\sigma) \in \Alg^{m-2}$, and hence $D^k(\sigma)\in \Alg^{m-2k}$.

Now for $\sigma=(w_+, w_-), \sigma'=(w_+^\prime, w_-^\prime) \in \Alg$  and $\re z \gg 0$,
\begin{multline*}
\varphi(z,\sigma \sigma') = \\ \Tr \Op^w(w_+)\Op^w(w_+^\prime)\ho^{-z}-(-1)^n\Tr \Op^w(w_-^\prime)\Op^w(w_-)\ho^{-z}=\\
\Tr \Op^w(w_+^\prime) (\ho^{-z}\Op^w(w_+)\ho^{z})\ho^{-z}- (-1)^n \Tr  (\ho^{z} \Op^w(w_-)\ho^{-z})\Op^w(w_-^\prime)\ho^{-z}=\\
\varphi(z, \sigma' \delta_z(\sigma))    
\end{multline*}

From the expansion \eqref{taylor}, for $\re z \gg 0$
\begin{equation*}
 \varphi(z,\sigma \sigma')  =
\sum_{k=0}^\infty \frac{z^k}{k!}\varphi(z, \sigma'D^k(\sigma)). 
\end{equation*}
Since both sides admit meromorphic extensions to $\CC$, this equality holds when either (hence both) sides are holomorphic, in particular at $z=0$.   Setting $z=0$ yields
\begin{equation*}
 \varphi(0,\sigma \sigma')= \varphi(0,\sigma' \sigma)
 \end{equation*}
This completes the proof.
 \end{proof}

\subsection{Calculations of $\tau$}
In the section we give an alternative characterization of the trace $\tau$, which is useful in calculations.

For $w \in \mcW$ with asymptotic expansion $w\sim\sum_{m=-l}^\infty  w_{m}$,  with $w_{m}$ homogeneous of degree $-m$, integrating  the terms with $-l\le m\le 2n$ gives  
\begin{equation*}
\int\limits_{\|x\|^2+\|\xi\|^2\le R^2} w(x, \xi) dx \, d\xi=c_{2n+l}R^{2n+l}+\cdots +c_1R+c_0\log{R}+f(R),  \qquad R>0
\end{equation*}
The constants $c_{2n-m}$ depends on $w_m$, and the remainder term $f(R)$ is a smooth function of $R$ that has a (finite) limit when $R \to \infty$.
The logarithmic term arises from integrating $w_{2n}$. Note that the homogeneous function $w_{2n}$ is not integrable on the ball $\|x\|^2+\|\xi\|^2\le R^2$,
unlike the other terms $w_m$ with $-l\le m\le 2n-1$.

We define the linear functional  
\begin{equation*}\widetilde{\Tr} \colon \mcW\to \CC \text{ by } \widetilde{\Tr}(w) :=(2 \pi)^{-n} \lim_{R\to \infty} f(R).\end{equation*}
If  $w\in \mcW^{-2n-1}$ and hence is of trace class then 
\begin{equation*}\widetilde{\Tr}(w)=\Tr(w)
\end{equation*}
Now if $\sigma=(w_+,w_-)\in \Alg$, then $w_+,w_-\in \mcW_q$, and if $w_+ \sim \sum_{m=-l}^\infty w_{2m}$ then $w_- \sim \sum_{m=-l}^\infty (-1)^m w_{2m}$, where $w_{2m}$ is homogeneous of degree $-2m$. 
For $(x, \xi) \ne (0, 0)$, let 
\begin{equation}\label{wpm}
\widetilde{w}_+:= w_+-\sum \limits_{m=-l}^n w_{2m} \qquad 
\widetilde{w}_-:= w_--\sum \limits_{m=-l}^n (-1)^m w_{2m} 
\end{equation}
and 
\begin{equation*}
\mathcal{R}(\sigma):=\widetilde{w}_+-(-1)^n \widetilde{w}_- = w_+-(-1)^n w_- -  2 \sum  w_{2m}
\end{equation*}
where the summation is over $m\in\ZZ $ such that $-l \le m \le n-1$ and $n+m$ is odd.

\begin{proposition}\label{tautilde2}
For $\sigma=(w_+, w_-)\in \Alg$,
\begin{equation*} \widetilde{\Tr}(w_+)-(-1)^n\widetilde{\Tr}(w_-)=(2 \pi)^{-n}\int \mathcal{R}(\sigma) dx\,d\xi\end{equation*}
\end{proposition}
\begin{proof}
The terms in $w_+$ and $w_-$ that are homogeneous of degree $-2n$ (i.e. with $m=n$) cancel out in $w_+-(-1)^nw_-$.
Thus, all terms $w_{2m}$ that appear in the formula for $\mathcal{R}(\sigma)$  are integrable,
and therefore  $\mathcal{R}(\sigma)$ is integrable as well (even though it is singular at $(x, \xi) =(0, 0)$).
\begin{equation*} \widetilde{\Tr}(w_+-(-1)^nw_-) = (2 \pi)^{-n}\int \mathcal{R}(\sigma) dx\,d\xi\end{equation*}
\end{proof}
The goal of this section is to derive the following formula for $\tau$.
\begin{proposition}\label{tautilde}
For $\sigma=(w_+, w_-)\in \Alg$,
\begin{equation*}
\tau(\sigma)= (2 \pi)^{-n}\int \mathcal{R}(\sigma) dx\,d\xi
\end{equation*}
\end{proposition}
Proposition \ref{tautilde} has some important corollaries.
\begin{corollary}\label{tau_of_polynomial}
Let $\sigma=(w_+, w_-)\in \Alg$ be such that $w_+(x,\xi), w_-(x,\xi)$ are polynomials in $(x,\xi)\in \RR^{2n}$. Then $\tau(\sigma)=0$.
\end{corollary}
\begin{proof}If $w(x,\xi)$ is a polynomial then $\widetilde{\Tr}(w)=0$. \end{proof}
\begin{corollary}\label{tau_invariance}
Let $\phi \in \mathrm{Sp}(\RR^{2n})$, $\sigma\in \Alg$. Then 
\begin{equation*}\tau( \phi(\sigma))= \tau(\sigma)\end{equation*}
\end{corollary}
\begin{proof}
A symplectic transformation $\phi$ preserves the measure $dx\,d\xi$.
\end{proof}

The remainder of this section is devoted to the proof of Proposition \ref{tautilde}.
The proof relies on the well-known connection between  zeta functions and  heat kernel expansions.

Given $\sigma=(w_+, w_-) \in \Alg$,
let $a:=w_+-(-1)^n w_-\in \mcW_q^{2l}$, so that
\begin{equation*} \tau(\sigma)=\Trh(a)\end{equation*}
The zeta function $\zeta_a(z)$ and the heat kernel are related by the Mellin transform.
With $A= \Op^w(a)$, we have :
\begin{equation*}
\Gamma(z) \Tr A \ho^{-z}= \int_0^\infty t^{z-1} \Tr A e^{-t\ho}dt \qquad \re z >0
\end{equation*}
$\Gamma(z)$ has  simple poles at nonpositive integers, and its  residue at $z=0$ is $1$. By Proposition \ref{antisymmetry} we have $\Res(a)=0 $ and hence $\Tr A \ho^{-z}$ is holomorphic at $z=0$.
As a consequence $\Gamma(z) \Tr A \ho^{-z}$ has at most simple poles at nonnegative  integers $z$, and at most  second order poles at  negative integers.
The inverse Mellin transform converts this information into the  
 asymptotic expansion of $\Tr A e^{-t\ho}$ 
\begin{equation}\label{expansion}
  \Tr A e^{-t\ho} \sim \sum \limits_{k=-n-l}^\infty a_k t^k+ \sum_{k=1}^\infty r_k t^k \log t  \text{ as $t \downarrow 0$.}
\end{equation}
The coefficients $a_k$, $k\le 0$, are the residues of $\Gamma(z) \Tr A \ho^{-z}$ at $z=-k$.
In particular
\begin{equation*}
a_0=  \Res|_{z=0} \,\Gamma(z) \Tr A \ho^{-z}=\Trh(a)  
\end{equation*}
In summary:
\begin{proposition}\label{heatkernel}
If  $\sigma=(w_+, w_-) \in \Alg$ then $\tau(\sigma)$ is equal to the constant term in the asymptotic expansion in powers of $t$ (for $t\downarrow 0$) of the expression
\begin{equation*}  \Tr \,(\Op^w(w_+)e^{-t\ho})-(-1)^n\Tr\,(\Op^w(w_-)) e^{-t\ho})\end{equation*}
\end{proposition}

We will need the following lemma.
\begin{lemma}\label{trab}
Let $a \in \mcW$, $b \in \mcS$. Then 
\begin{equation*} \Tr \Op^w(a\# b) =\Tr \Op^w(ab)\end{equation*}
where $ab$ is the pointwise product of functions $(ab)(x,\xi)=a(x,\xi)b(x,\xi)$.
\end{lemma}
\begin{proof}
The Schwarz kernel of $\Op^w(a) \Op^w(b)$ is the oscillatory integral
\begin{equation*}
K(x,y)=(2 \pi)^{-2n}\int a\left(\frac{x+z}{2}, \zeta\right)b\left(\frac{y+z}{2}, \tau\right)e^{i(x-z) \cdot \zeta-i (y-z) \cdot \tau}  d\zeta\, d \tau\, dz
\end{equation*}
With $x=y$ and the change of variables $(u,v) = ((x+z)/2, x-z)$ we find
\begin{align*}
\Tr \Op^w(a) \Op^w(b) = \int K(x,x) dx &= 
(2 \pi)^{-2n}\int a\left(u, \zeta\right)b\left(u, \tau\right)e^{iv \cdot (\zeta-\tau)}  d\zeta\, d \tau\, du\,dv\\
&= (2 \pi)^{-n}\int a(u, \zeta)b(u, \zeta) du \, d\zeta
\end{align*}
\end{proof}

\begin{proof}[Proof of Proposition \ref{tautilde}]
By Proposition \ref{heatkernel} and  Lemma \ref{trab}, $\tau(\sigma)$ is equal to the constant term in the 
asymptotic expansion in powers of $t$ of the integral
\begin{equation*}
(2\pi)^{-n} \int (w_+-(-1)^nw_-) h_t dx \, d\xi
\end{equation*}
 with $h_t$ as in (\ref{Mehler}).
This expression is equal to
\begin{equation*}
(2\pi)^{-n} \int \mathcal{R}(\sigma) h_tdx \, d\xi+2(2\pi)^{-n}  \sum \int w_{2m} h_t dx \, d\xi
\end{equation*}
where the summation is over $m \in \ZZ$ such that $-l \le m \le n-1$ and $n+m$ is odd.
From  Mehler's formula (\ref{Mehler})  we see that $h_t(x,\xi)$ converges to $1$ as $t\downarrow 0$, uniformly on compact subsets of $\RR^{2n}$.
Since $\mathcal{R}(\sigma)$ is integrable, we obtain
\begin{equation*}\lim_{t\downarrow 0} \int \mathcal{R}(\sigma) h_t \, dx \, d\xi = \int \mathcal{R}(\sigma)\, dx \, d\xi\end{equation*}
Since $w_{2m}$ is homogeneous of degree $-2m$, we get
\begin{equation*}
\int w_{2m} h_t dx \, d\xi= \frac{c_m}{(\cosh t)^n}\int_0^{\infty} \rho^{2m-2n-1} e^{-\rho^{-2}\tanh t} d\rho\\
=\frac{c_m \Gamma(n-m)}{2(\sinh t)^{n-m}(\cosh t)^m}
\end{equation*}
for a constant $c_m$ which depends on $w_{2m}$. 
For  odd $n+m$ (and so odd $n-m$)
the function on the right hand side is odd, and hence the constant coefficient in its Laurent expansion at  $t=0$ vanishes. 
We conclude that
\begin{equation*}\tau(\sigma)=(2\pi)^{-n} \int \mathcal{R}(\sigma)\, dx \, d\xi\end{equation*}
\end{proof}


\subsection{Continuity of $\tau$}

We identify the  algebra of order zero Heisenberg principal symbols $\Symb$
 with the order zero subalgebra $\Alg^0_H$ of $\Alg_H$. 
As a vector space $\Symb$ is the space of smooth functions $C^\infty(S^*M)$,
which  is a Fr\'echet space in the usual way.

\begin{proposition}\label{tau_continuous}
The trace $\tau \colon \Alg^0\to \CC$ is continuous as a linear functional on the Fr\'echet space $C^\infty(S^{2n})$. 
\end{proposition}
\begin{proof}
With  $\sigma=(w_+, w_-)\in \Alg^0$, we need to give an upper bound for  $\int \mathcal{R}(\sigma) dx\, d\xi$
by a linear combination  of continuous  seminorms of $C^\infty(S^{2n})$.
Here we shall think of  $\sigma$ as a single smooth function on $S^{2n}$. 
We  use the coordinate $\rho_q$ near the equator, which is related to $(x,\xi)$ via $\rho_q=\pm \rho^2=\pm (\|x\|^2+\|\xi\|^2)$ (see section \ref{sec:Hsymb}).

In the region $\rho\le 1$  the asymptotic expansions 
\begin{equation*}w_+ \sim \sum_{m=0}^\infty w_{2m}\qquad w_- \sim \sum_{m=0}^\infty (-1)^m w_{2m}\end{equation*}
correspond to Taylor series of $\sigma$ near the equator of $S^{2n}$ in powers of $\rho_q$.
We therefore obtain estimates for the remainders $\widetilde{w}_\pm$ defined in \eqref{wpm},
\begin{equation*}
|\widetilde{w}_\pm(x, \xi)| \le M\frac{|\rho_q|^{n+1}}{(n+1)!} \qquad M=\sup_{\rho_q \le 1} |\partial_{\rho_q}^{n+1}\sigma|
\end{equation*}
From  $\mathcal{R}(\sigma)=\widetilde{w}_+-(-1)^n\widetilde{w}_-$ we obtain
\begin{equation*}
\left| \int_{\rho \le 1} \mathcal{R}(\sigma) dx\, d\xi  \right|
 \le C_n \sup_{\rho_q \le 1} |\partial_{\rho_q}^{n+1}\sigma|
 \end{equation*}
For the  integral over the region $\rho\ge 1$, we write
\begin{equation*} \left| \int_{\rho \ge 1} \mathcal{R}(\sigma) dx\, d\xi \right|\le   \int_{\rho \ge 1} |w_+(x,\xi)| dx\, d\xi +
 \int_{\rho \ge 1} |w_-(x,\xi)| dx\, d\xi + 2 \sum \int_{\rho \ge 1} |w_{2m}(x,\xi)| dx\,d\xi\end{equation*}
(with summation over $0\le m\le n-1$ and $m+n$ odd).
The first two terms are bounded by a  multiple of the supremum norm $\|\sigma\|_\infty$.
Finally, for the homogeneous terms $w_{2m}$ we let
\begin{equation*}w_{2m}(x,\xi)=f_{2m}(\theta)\rho^{2m}\end{equation*}
where $f_{2m}(\theta)=w_{2m}(\rho x, \rho \xi)$ is a smooth function of $\theta=(\rho x,\rho \xi)\in S^{2n-1}$.
We obtain estimates
\begin{equation*}\int_{\rho \ge 1} |w_{2m}(x,\xi)| dx\,d\xi\le D_{n,m} \sup_{\theta\in S^{2n-1}} |f_{2m}(\theta)|\end{equation*}
The supremum of $f_{2m}$ is a continuous seminorm of $C^\infty(S^{2n})$ because  $f_{2m}$
is equal to the restriction to the equator of the derivative $\partial^m_{\rho_q}\sigma$.

\end{proof}

Let $\tilde{\rho}(x,\xi)>0$ be a strictly positive smooth function  on $\RR^{2n}$,
which is equal to $ \rho(x,\xi)=(\|x\|^2+\|\xi\|^2)^{-1/2}$ if  $\rho\le 1$.
Then the map
\[ \Alg^{2l}\to \Alg^0\qquad (w_+,w_-)\mapsto (\tilde{\rho}^{2l}w_+, (-1)^l\tilde{\rho}^{2l}w_-)\]
is a linear isomorphism of vector spaces  $\Alg^{2l}\cong C^\infty(S^{2n})$.
Via this identification, the continuous seminorms of $C^\infty(S^{2n})$ determine seminorms on $\Alg^{2l}$,
making $\Alg^{2l}$ into a Fr\'echet space.
The proof of Proposition \ref{tau_continuous}, with minor changes, shows that $\tau \colon \Alg^{2l}\to \CC$ is continuous.

With the Fr\'echet space structure as discussed above, we have continuous injections
\[ \Alg^0\to \Alg^2\to \Alg^4\to \cdots\]
The algebra $\Alg$ can be given the direct limit topology, making it into an LF space.
Then $\tau \colon \Alg\to \CC$ is continuous.

\subsection{The trace $\tau$ and the connection}

By Corollary \ref{tau_invariance} and Proposition \ref{tau_continuous}, the trace $\tau$ extends to a linear map
\[ \tau \colon \Alg_H=C^\infty(P_H;\Alg)^{Sp(2n)}\to C^\infty(M)\]
with the property
\[ \tau(ab)=\tau(ba)\qquad a,b\in \Alg_H\]
Since the continuous linear map $\tau \colon \Alg^{2l}\to \CC$ is automatically smooth,
it follows that if $\sigma\in \Alg_H$ then $\tau(\sigma)\in C^\infty(M)$.
Likewise, we have 
\begin{equation*}
\tau  \colon  \Omega^k(P_H; \Alg)_{basic} \to \Omega^k(P_H)_{basic} 
\end{equation*}
or more succinctly,
\begin{equation*}
\tau  \colon  \Omega^k(\Alg_H)\to \Omega^k(M)
\end{equation*}
Then $\tau$ is a graded trace,
\begin{equation*}
\tau(\eta_1\eta_2)=(-1)^{k_1+k_2}\tau(\eta_2\eta_1) \qquad  \eta_j \in \Omega^{k_j}(\Alg_H)
\end{equation*}
\begin{lemma} For all $\eta \in \Omega^k(M, \mathcal{A}_H)$, 
\begin{equation*}\tau(\con(\eta)) = d \tau (\eta)
\end{equation*}
\end{lemma}
\begin{proof}
\begin{equation*}\tau(\con (\eta)) =\tau(d\eta +[\Iso(\alpha), \eta])=\tau(d \eta)=d \tau(\eta).\end{equation*}
\end{proof}

\begin{lemma}\label{tau_of_curv}
For every non-negative integer $k=0,1,2,\dots$,
\[ \tau(\curv^k)=0\]
\end{lemma}
\begin{proof}
By definition, $\curv=\nu(\theta)$ is an element in $\Alg_H$
which restricts in each fiber $\Alg(H_p^*,\omega_p)$ to an element $(w_+,w_-)$,
where $w_+$ and $w_-$ are polynomials of degree 2.
The lemma then follows from Corollary \ref{tau_of_polynomial}.

\end{proof}

\section{A generalized cycle and its character}\label{section_character}

The goal of this section is to construct the homomorphism 
\[\chi  \colon  K_1(\Symb) \to H^{odd}(M)\]
appearing in our index formula. The construction uses the connection, curvature and trace constructed in the preceding sections and relies on  techniques from cyclic (co)homology theory. Therefore  we give a very brief review of definitions from cyclic homology theory in Section \ref{cyclic section} and of cycles and generalized cycles in Section \ref{section cycles}. In Section \ref{section character map} we spell out the construction of the character map in a general geometric context, and in Section \ref{section character2} we specialize it to our situation.

\subsection{The Chern character in cyclic homology}\label{cyclic section}
In this section we give a very brief overview of the periodic cyclic homological complex, mostly to fix the notations. 
The standard reference for this material is \cite{loday}.

For a complex unital algebra $A$ set $C_l(A):= A \otimes(A/(\mathbb{C} \cdot 1))^{\otimes l}$, $l \ge 0$.
One defines differentials $b \colon  C_l(A) \to C_{l-1}(A)$ and $B  \colon  C_l(A) \to C_{l+1}(A)$ by
\begin{equation*}
b (a_0\otimes a_1\otimes \ldots a_l):= \sum_{i=0}^{l-1} (-1)^ia_0\otimes \ldots a_ia_{i+1}\otimes \ldots a_l+(-1)^la_la_0\otimes a_1\otimes \ldots a_{l-1}
\end{equation*}
\begin{equation*}
B (a_0\otimes a_1\otimes \ldots a_l):= \sum_{i=0}^{l} (-1)^{li}1 \otimes a_i\otimes a_{i+1}\otimes \ldots a_{i-1}\text{ (with $a_{-1}:=a_l)$}
\end{equation*}
One verifies directly that $b$, $B$ are well defined and satisfy $b^2=0$, $B^2=0$, $Bb+bB=0$.  Let $u$ be a formal variable of degree $-2$. 
The space of periodic cyclic chains of degree $ i \in \ZZ$ is defined by
\begin{equation*}
CC^{per}_{i}(A) = \left(C_{\bullet}(A)[u^{-1},u]]\right)_i=\prod \limits_{-2n+l=i} u^nC_l(A).
\end{equation*}
Note that $CC^{per}_{i}(A)= uCC^{per}_{i+2}(A)$.
We will write a chain in $CC^{per}_i(A)$ as
$ \alpha = \sum \limits_{i+2m\ge 0}u^m \alpha_{i+2m}$ where $\alpha_l \in C_l(A)$.
The boundary is given by $b+uB$ where $b$ and $B$ are the Hochschild and Connes boundaries of the cyclic complex. 
The homology of this complex is periodic cyclic homology, denoted $HC^{per}_{\bullet}(A)$.

If $r \in M_n(A)$ is invertible the following formula defines a cycle in the periodic cyclic complex:
\begin{equation}\label{cyclicchernformula}
\Ch(r):= -\frac{1}{2 \pi i} \sum_{l=0}^{\infty}(-1)^l\, l!\, u^l \tr (r^{-1}\otimes r)^{\otimes (l+1)} \in CC^{per}_1(A)
\end{equation}
where $\tr  \colon  (A \otimes M_n(\mathbb{C}))^{\otimes k} \to A^{\otimes k}$ is the map given by
\begin{equation*}
\tr (a_0\otimes m_0)\otimes (a_1\otimes m_1) \otimes \ldots (a_k\otimes m_k)= (\tr m_0m_1\ldots m_k) a_0  \otimes a_1 \otimes \ldots a_k
\end{equation*}
In the case of interest to us,  $A$ will be a Fr\'echet algebra.
In that case we will use a projective tensor product to define $C_l(A)=A \otimes(A/(\mathbb{C} \cdot 1))^{\otimes l}$.
One can define the topological $K$-theory of a Fr\'echet algebra as $K_1(A):=\pi_0(GL(A))$. 
With these definitions, \eqref{cyclicchernformula} defines  the (odd) Chern character homomorphism
\begin{equation}\label{CyclicChern}
\Ch  \colon  K_1(A) \to HC^{per}_1(A)
\end{equation}
from  topological $K$-theory to  periodic cyclic homology.

\subsection{Cycles and characters}\label{section cycles}
As in \citelist{\cite{Co85} \cite{Co94}},  a {\em cycle} of dimension $n$ is a triple $(\Omega, d, \int)$ where $\Omega=\bigoplus_{j=0}^n\Omega^j$ is a graded complex algebra, $d \colon \Omega\to\Omega$ is a graded derivation of degree 1 such that $d^2=0$, and $\int \colon \Omega^n\to \CC$ is a  graded trace on $\Omega$ such that $\int d\beta=0$ for all $\beta\in \Omega^{n-1}$.
If $A$ is a complex algebra, then a cycle over $A$ is given by a cycle $(\Omega, d, \int)$ together with  a homomorphism $\rho \colon  A\to \Omega^0$.
The {\em character} of such a cycle is the $(n+1)$-linear map
\[ (a_0,a_1,\dots,a_n)\mapsto \int \rho(a_0)\,d(\rho(a_1))\,d(\rho(a_2))\,\cdots\, d(\rho(a_n))\qquad a_j\in A\]
Elements in cyclic cohomology $HC^\bullet(A)$ can be represented as cycles over $A$. 

In \cites{gor1, gor2} the first author considered the notion of {\em generalized cycle} of degree $n$ given by a quadruple $(\Omega^\bullet, \nabla, \theta, \int)$ where
\begin{itemize}
    \item $\Omega^\bullet$ is a graded algebra
    \item $\nabla \colon \Omega^\bullet \to \Omega^{\bullet+1} $ is a graded derivation 
    \item $\theta \in \Omega^2$ is such that 
    $ \nabla^2(\beta)=\theta\beta-\beta\theta  \text{ for } \beta\in \Omega \text{ and } \nabla(\theta)=0$. 
    In other words, $(\Omega^\bullet, \nabla, \theta)$ is a curved differential graded algebra
    \item Finally $\int  \colon  \Omega^n\to \CC$ is a graded trace such that  $\int \nabla\beta=0$ for $\beta\in \Omega$
\end{itemize}
With every generalized cycle explicit formulas of \cites{gor1, gor2}  associate its character, which is a cocycle in the cyclic $b$, $B$ bicomplex. This construction extends the construction of the character of a cycle.




These formulas apply in our context as follows.
The triple $(\Omega^\bullet \Alg_H,\con,\curv)$ is a curved dga.
If we let 
\[ \stackinset{c}{}{c}{}{-\mkern4mu}\int \beta := \int_M \tau(\beta) \wedge 
\hat{A}(M)\qquad \beta\in \Omega^\bullet\Alg_H\]
then the quadruple 
\[(\Omega^\bullet\Alg_H,
\con,\curv,\displaystyle\stackinset{c}{}{c}{}{-\mkern4mu}\int)\]
is a generalized cycle (or, more correctly, a finite sum of generalized cycles),
thanks to the results of  sections  \ref{section_curvature} and \ref{section_trace}.
With the inclusion $\Symb\cong \Alg_H^0\subset \Alg_H$,
this defines a generalized cycle over the algebra of principal Heisenberg symbols $\Symb$.
The formalism of \cites{gor1, gor2} then gives us a periodic cyclic cocycle in $ CC^1_{per}(\Symb)$ which is continuous on the Fr\'echet algebra $\Symb$ (see Proposition \ref{tau_continuous} and the Appendix for a brief discussion of continuity).

Therefore we can pair this  cyclic cocycle  with topological $K$-theory, and we obtain a map
\[ K_1(\Symb)\stackrel{\Ch}{\lra} HC^{per}_1(\Symb)\to \CC\]  
We shall prove that this is the index map for the Heisenberg calculus.

\subsection{The character map} \label{section character map}
Let $M$ be a closed smooth manifold, and let $A$ be a unital complex algebra equipped with an algebra homomorphism $C^\infty(M)\to Z(A)$, where $Z(A)$ is  the center of $A$. Denote
\[ \Omega^j(A) :=  \Omega^j(M)\otimes_{C^\infty(M)} A \qquad \Omega^\bullet(A):= \bigoplus_{j} \Omega^j(A)\]
Assume that the following data is given:
\begin{itemize}
    \item A connection $\nabla  \colon  A \to \Omega^1(A)$ with $\nabla 1=0$ and
    \begin{equation*}
    \nabla(fa) = df\otimes a+ f \nabla(a) \qquad f\in C^\infty(M),\,a\in A
    \end{equation*}
    which acts as a derivation of $A$:
    \begin{equation}\label{defN}
    \nabla(ab) = \nabla(a) b+ a \nabla(b) \qquad a, b \in A
    \end{equation}
    $\nabla$ extends to a graded derivation of $\Omega^\bullet(A)$ of degree 1 in the standard way.
    \item An element $\theta \in \Omega^2 (A)$ such that 
    \begin{equation}\label{defTe}
    \nabla^2(a) =[\theta, a]\qquad  \nabla(\theta)=0
    \end{equation}
    \item A $C^\infty(M)$-linear trace $\tau  \colon  A \to C^\infty(M) $ satisfying 
    \begin{equation*}\tau(ab) = \tau(ba)
    \end{equation*}
    $\tau$ extends to a map $\Omega^\bullet(A)  \to \Omega^\bullet(M)$, and we assume that 
    \begin{equation}\label{defTa}
    \tau(\nabla a) = d \tau(a)
    \end{equation}
\end{itemize}
If  $\mathcal{C}$ is a closed de Rham current $\mathcal{C}$ on $M$, then the quadruple $(A,\nabla,\theta,\mathcal{C}\circ \tau)$ is a generalized cycle, as above.
However, it is unnecessary to bring in de Rham currents at this point.
Instead, one can   construct the following morphism of complexes.
\begin{theorem}[\cites{gor1, gor2}]\label{character map}
 The map 
\begin{equation}\label{defT}
    T  \colon  CC^{per}_{\bullet}(A) \to \left( \Omega^\bullet(X)[u^{-1}, u], ud \right)
\end{equation} 
given by
\begin{multline*}
T(a_0\otimes a_1 \ldots \otimes a_k) = 
 \sum \limits_{i_0, \ldots, i_k \ge 0}\frac{(-1)^{i_0+\ldots +i_k}}{(i_0+i_1+\ldots i_k+k)!} \tau \left(  a_0 (u\theta)^{i_0} \nabla(
a_1)(u\theta)^{i_1}\dots \nabla(a_k) (u\theta)^{i_k}\right)
\end{multline*}
is a morphism of complexes.
\end{theorem} 
\begin{remark}
A natural framework for such identities in cyclic cohomology is provided
by the theory of operations on cyclic cohomology of Nest and Tsygan, cf. \cites{nt1, nt2, nt3}.
\end{remark}

For a given algebra $A$ one can modify $\nabla$ and $\theta$ (without changing $\tau$) as follows. Let $\kappa \in \Omega^1(A)$. Set
\begin{equation}\label{perteta}
\nabla':= \nabla+[\kappa, \cdot], \ \theta':= \theta +\nabla \kappa +\kappa^2
\end{equation}
Then $\nabla', \theta', \tau$ satisfy all the conditions above and let $T'$ be the corresponding morphism of complexes. We then have the following.
\begin{proposition}[\cites{gor1, gor2}]\label{HomotopyT}
The morphisms $T$ and $T'$ are chain homotopic. 
\end{proposition}

Assume now that $A$ is a Fr\' echet algebra, and that  $C^\infty(M)\to Z(A)$, $\nabla \colon  A\to \Omega^1(A)$, $\tau \colon  A\to C^\infty(M)$ are continuous. 
Then the morphism $T$ from Theorem \ref{character map} is continuous and we can define the character map
\begin{equation*}
\chi  \colon  K_1(A) \to H^{odd}(M)
\end{equation*}
as the composition
\begin{equation*}
  K_1(A) \overset{\Ch}{\longrightarrow} HC_1(A) \overset{T}{\longrightarrow} H^{odd}(M)[u^{-1}, u] \overset{R}{ \longrightarrow} H^{odd}(M)
\end{equation*}

\begin{itemize}
    \item $\Ch$ is the Chern character in cyclic homology from \eqref{CyclicChern}
    \item $T$ is the map in (co)homology determined by the morphism \eqref{defT}
    \item $R$  is $\CC$-linear map defined by  ``evaluation'' of the formal variable $u$,
    \[ R(u) = \frac{1}{2\pi i}\]
    in other words, $R(\sum u^q \alpha_{q}) = \sum (2 \pi i)^{-q}\alpha_{q}$, $\alpha_{q} \in H^\bullet(M)$
\end{itemize}

Explicitly, for a class $[r]$ in $K_1(A)$ represented by an invertible $r \in M_n(A)$, its image under $\chi$ is represented by the differential form

\begin{equation}\label{defchi}
\chi(r) = \sum \limits_{l\ge 0} \sum \limits_{i_0, \ldots, i_{2l+1} \ge 0} \left(\frac{-1}{2 \pi i}\right)^{I+l+1} \frac{l!}{(I+2l+1)!} \tau \left(\tr  r^{-1} \theta^{i_0} \nabla(
r)\theta^{i_1}\nabla(r^{-1}) \dots \nabla(r) \theta^{i_{2l+1}}\right)
\end{equation}
where
\begin{equation*}
I=i_0+i_1+\ldots +i_{2l+1}
\end{equation*}
and $\tr  \colon  M_n(\mathbb{C}) \otimes \Omega^\bullet(A) \to \Omega^\bullet(A)$ is the matrix trace.  
By Proposition \ref{HomotopyT} the cohomology class of $\chi(r)$  does not change if $\nabla$, $\theta$ are replaced by $\nabla'$, $\theta'$ as in \eqref{perteta}.

\begin{remark}
For a closed de Rham current $\mathcal{C}$ on $M$, the quadruple $(A,\nabla,\theta,\mathcal{C}\circ \tau)$ is a generalized cycle.
The character of $(A,\nabla,\theta,\mathcal{C}\circ \tau)$ (as defined by Connes)
determines a map $K_1(A)\to \CC$.
This map is equivalent to the composition $\mathcal{C}\circ \chi$
(See \cites{gor1, gor2}).
\end{remark}

\begin{example}
Let $E$ be a vector bundle over $M$, and let $A= \Gamma(M; \End(E))$.  
A choice of connection $\nabla$ on $E$  defines a derivation on $A$ as in \eqref{defN}. 
The curvature of $\nabla$ is an element $\theta \in \Omega^2(M; \End(E))$, which satisfies  equations \eqref{defTe}. Finally, the fiberwise trace $\tr \colon  \End(E_p)\to\CC$ determines a trace $\tau \colon  A \to C^\infty(X)$ 
that satisfies \eqref{defTa}. For an invertible element $r \in A$ we therefore obtain a differential form $\chi(r)$ defined by \eqref{defchi},  and a cohomology class $[\chi(r)] \in H^{odd}(M)$.
\begin{proposition}\label{chern}
Assume that $M$ is compact, and let  $[r]$ denote the class of  $r$ in $K_1(A) \cong K^1(M)$.  Then
\begin{equation*}
[\chi(r)] = \Ch[r] \in H^{odd}(M).
\end{equation*}
\end{proposition}
\begin{proof}
Consider first the case when the bundle $E$ is trivial, with an arbitrary connection $\nabla$ with curvature $\theta$. In this case we can replace $\nabla$, $\theta$ by de Rham differential $d$ and $0$ respectively without changing the cohomology class of $\chi(r)$. With this choice of connection 
\begin{multline*}
    \chi(r) = \sum \limits_{l\ge 0}   \left(\frac{-1}{2 \pi i}\right)^{l+1} \frac{l!}{(2l+1)!} \tr  r^{-1}  dr d(r^{-1})\ldots d(r^{-1})  dr =\\
-\sum \limits_{l\ge 0}   \frac{1}{(2 \pi i)^{l+1}} \frac{l!}{(2l+1)!} \tr  (r^{-1}  dr)^{2l+1} 
\end{multline*}
and the last expression is the well-known formula for the Chern character. 
This proves the result for the trivial bundle.
Now for a general bundle $E$ find $E'$ such that $E\oplus E'$ is trivializable. Choose any connection $\nabla'$ on $E'$ and endow  $E\oplus E'$ with the connection $\nabla \oplus \nabla'$. Then $\chi(r \oplus \id) =\chi(r)$. On the other hand $[\chi(r \oplus \id)] =\Ch([r \oplus \id]) = \Ch([r])$, which complets the proof in the general case.

\end{proof}
\end{example}

\subsection{The character map for $\Symb$}\label{section character2}

All the data needed to construct the character map is present for the algebra $\Symb$ of principal Heisenberg symbol of order zero.
We have a curved dga $(\Omega^\bullet(\Alg_H),\con, \curv)$,
a closed graded trace $\tau \colon \Omega^\bullet(\Alg_H)\to \Omega^\bullet(M)$,
and a homomorphism (an inclusion) $\Symb\to \Alg_H$.
We summarize our conclusion in the following theorem.

\begin{theorem}\label{summarychi}
Let $\sigma$ be an invertible element in $M_r(\Symb)$.
Choose a symplectic connection $\nabla$ on $H$ and let $\con$ and $\curv$ be as in \eqref{defcon}, \eqref{deftheta}.
Then the formula
 \begin{equation}
\chi(\sigma) = \sum \limits_{l\ge 0} \sum \limits_{i_0, \ldots, i_{2l+1} \ge 0} \left(\frac{-1}{2 \pi i}\right)^{I+l+1} \frac{l!}{(I+2l+1)!} \tau \left(\tr  \sigma^{-1} \curv^{i_0} \con(\sigma)\curv^{i_1}\con(\sigma^{-1}) \dots \con(\sigma) \curv^{i_{2l+1}}\right),
\end{equation}
\begin{equation*}
I=i_0+i_1+\ldots i_{2l+1}
\end{equation*}
defines a homomorphism \[\chi  \colon  K_1(\Symb) \to H^{odd}(M)\] This homomorphism is independent of the choice of connection $\nabla$.
\end{theorem}
The independence of $\nabla$ is immediate from Proposition \ref{HomotopyT}, the conditions of which are satisfied according to Lemma \ref{changeofconcurv}.

\section{Toeplitz operators}\label{section_Toeplitz}

A  Toeplitz operator is an order zero pseudodifferential operator in the Heisenberg calculus.
We calculate the character $\chi(\sigma)$  of the symbol of  a Toeplitz operator $T_f$,
\[ \chi(\sigma_H(T_f)) =  \Ch(f)\wedge \exp\left(\frac{1}{2}c_1(H^{1,0})\right)\]
This calculation establishes the equivalence of our index formula with that of Boutet de Monvel in the case of Toeplitz operators.

\subsection{Boutet de Monvel's theorem}

Let $\tilde{N}$ be a complex analytic manifold of complex dimension $n+1$, and $N\subset \tilde{N}$ a relatively compact open submanifold with  smooth boundary $M=\partial N$ of real dimension $2n+1$.
Let $H^{1,0}\subset TM\otimes \CC$ be the complex vector bundle of holomorphic tangent vectors on $\tilde{N}$ that are tangent to $M$,
\[H^{1,0} := T^{1,0}\tilde{N}|M\cap (TM\otimes \CC)\] 
Choose a defining function of the boundary $r \colon N\to \RR$ with $N=r^{-1}((-\infty,0))$, $M=r^{-1}(0)$, and $dr\ne 0$ on  $M$.
The {\em Levi form} is the hermitian form on the fibers of $H^{1,0}$ defined by
\[ \ang{v,w} := \partial\bar{\partial}r(v,\bar{w})\qquad v,w\in H^{1,0}_p\]
The boundary of $N$  is called  {\em strictly pseudoconvex} if the Levi form is strictly positive.
This definition is independent of the choice of $r$. Strict pseudoconvexity is biholomorphically invariant.
A domain $N\subset \CC^{n+1}$ with a smooth boundary that is strictly convex in the Euclidean sense is strictly pseudoconvex.
In $\CC^{n+1}$, a strictly pseudoconvex domain is the same as a domain of holomorphy.

Let $\alpha$ be the restriction of the $(1,0)$-form $-i\partial r$ to  $M$.
Because $dr(v)=\partial\gamma(v)+\bar{\partial}\gamma(v)=0$ for $v\in TM\otimes \CC$, $\alpha$ is a real 1-form.
A holomorphic vector $v\in T^{1,0}\tilde{N}|M$ is tangent to $M$ if $dr(v) = \partial r(v) =0$.
Since  $\bar{\partial}r(v)=0$ this is equivalent to $\alpha(v+\bar{v})=0$.
We have a canonical isomorphism of {\em real} vector bundles
$T^{1,0}\tilde{N}\cong T\tilde{N} \colon v\mapsto v+\bar{v}$.
Then $H^{1,0}$ is identified with the real vector bundle $H\subset TM$
of (real) tangent vectors that are annihilated by $\alpha$.
Note that $d\alpha$ is the restriction of $i\partial\bar{\partial}\gamma$ to $M$. Strict positivity of the Levi form implies that $d\alpha$ is nondegenerate (i.e. symplectic) when restricted to $H$.
Thus, $\alpha$ is a contact form on $M$.

The Hardy space  $H^2(M)$ is the space of $L^2$-functions on $M$ that extend to a holomorphic function on $N$.
 The Szeg\"o projection $S$ is the orthogonal projection 
 \[ S \colon L^2(M)\to H^2(M)\]
For a continuous map $f \colon M\to \mathrm{GL}(r,\CC)$, let $\mathcal{M}_f$ be the corresponding multiplication operator 
on  $L^2(M)\otimes \CC^r$.
The Toeplitz operator $T_f$ is the composition of $\mathcal{M}_f$
with $S\otimes I_r$,
\[ T_f =(S\otimes I_r)\mathcal{M}_f  \colon  H^2(M)\otimes \CC^r\to H^2(M)\otimes \CC^r\]
$T_f$ is a bounded Fredholm operator.
\begin{theorem}\label{BdM}[Boutet de Monvel \cite{Bo79}]
\[  \ind  T_f = \int_M \Ch(f)\wedge \Td(H^{1,0})\]
\end{theorem}
Here $M$ is oriented as the boundary of $N$ by the `outward normal first' convention: a frame of tangent vectors $(v_1,v_2,\dots, v_{2n+1})$ on $M$ is positively oriented if $({\bf n}, v_1,v_2,\dots, v_{2n+1})$ is positively oriented on $\tilde{N}$, where ${\bf n}$ is an outward pointing normal vector, $dr({\bf n})>0$.
Equivalently, the volume form $\alpha(d\alpha)^n$ on $M$ is positively oriented.

\subsection{The Szeg\"o projection in the Heisenberg calculus}
A reference for  the material in this section is \cite{FS74}.

With $M=\partial\tilde{N}$ as above,
let the tangential Cauchy-Riemann operator
\[\bar{\partial}_b \colon C^\infty(M)\to \Gamma((H^{0,1})^*)\]
be defined by $\ang{\bar{\partial}_bf,W}=\ang{df,W}$
for $W\in \Gamma(H^{0,1})$.
Note that $\bar{\partial}_bf=0$ if $W.f=0$ for all anti-holomorphic vector fields $W$ on $\tilde{N}$ that are tangent to $M$.
Such a function extends to a holomorphic function on $\tilde{N}$, at least  in a neighborhood of $M$.

If we impose a Hermitian metric on $M$, 
we can form the formal adjoint $\bar{\partial}_b^*$.
The {\em Kohn Laplacian} is the second order operator   $\Box_b:= \bar{\partial}_b^*\bar{\partial}_b$.
The projection onto the kernel of $\Box_b$ differs from the Szeg\"o projection $S$ by a finite rank smoothing operator.

The Szeg\"o projection $S$ is not a classical pseudodifferential operator,
but it is an order zero pseudodifferential operator in the Heisenberg calculus.
It has the same principal Heisenberg symbol as the projection onto the kernel of $\Box_b$.

Up to  order 1 terms in the Heisenberg calculus,
$\Box_b$ is equal to the operator
\[ \mathscr{L}_n = \Delta_H+inT\]
where $\Delta_H$ is a sublaplacian, and $T$ is the Reeb vector field.
The principal Heisenberg symbol (of order 2) of $\Delta_H$
is $\sigma_H(\Delta_H) = (Q,Q)$, where $Q(x,\xi)=\|x\|^2+\|\xi\|^2$ is the harmonic oscillator.
In the Heisenberg calculus, the Reeb vector field $T$
is an order 2 operator with principal symbol $\sigma_H(T) = (i,-i)$.
Thus, 
\[\sigma_H(\mathscr{L}_n)=(Q-n,Q+n)\]
The kernel of $Q-n$ is spanned by the vacuum vector of the harmonic oscillator. 
$Q+n$ is strictly positive, and has no kernel.
Thus,  the principal Heisenberg symbol  of the Szeg\"o projection is 
\[\sigma_H(S)=(s,0)\in \Symb\]
where $s\in \mcW(H_p^*,\omega_p)$ is the projection onto the vacuum of the harmonic oscillator.
From \eqref{Mehler} one derives the formula
\begin{equation}\label{szego_symbol}
s(v) := 2^ne^{-\|v\|^2}\qquad v\in H_p^*
\end{equation}

\subsection{Toeplitz operators on contact manifolds}
In \cite{EMxx}, Epstein and Melrose show how to generalize Boutet de Monvel's theorem to contact manifolds. 
If $M$ is a closed contact manifold with contact form $\alpha$, and $H=\Ker\alpha\subset TM$,
one can choose a complex structure $J \colon H\to H$, $J^2=-I$, that is compatible with the symplectic form $-d\alpha$.
Then \eqref{szego_symbol} defines a projection $s\in \Symb$.
Choose an arbitrary pseudodifferential operator of order zero $\tilde{S}\in \Psi_H^0(M)$ in the Heisenberg calculus
with symbol $s$. 
Since $s^2-s=0$, we see that $\tilde{S}^2-\tilde{S}\in \Psi^{-1}_H(M)$ is a compact operator.
Because the Heisenberg algebra $\Psi^0_H(M)$ is holomorphically closed,
one can then form a projection $S\in \Psi^{0}_H(M)$ with the same symbol as $\tilde{S}$.
Such a projection with symbol \eqref{szego_symbol} is called a generalized Szeg\"o projection. 

As shown in \cite{EMxx}, Boutet de Monvel's theorem generalizes to contact manifolds.
(See also \cite{BvE14}.)

\begin{theorem}\label{EM_Toeplitz}\cite{EMxx}
If $M$ is a closed contact manifold with generalized Szeg\"o projection $S$,
and 
\[T_f=(S\otimes I_r)\mathcal{M}_f+(I-S) \colon L^2(M)\otimes \CC^r\to L^2(M)\otimes \CC^r\] 
is the Toeplitz operator associated to $f \colon M\to \mathrm{GL}(r,\CC)$, then  
\[  \ind  T_f = \int_M \Ch(f)\wedge \Td(H^{1,0})\]
Here $M$ is oriented by the volume form $\alpha(d\alpha)^n$.
\end{theorem}

\subsection{Some calculations}

Let $V=\RR^{2n}$ with coordinates $(x,\xi)$ and symplectic form $\omega=\sum dx_j\wedge d\xi_j$.
We  identify 
\[ \RR^{2n}=\CC^n\quad (x_1,\dots, x_n,\xi_1,\dots, \xi_n) \mapsto (x_1+i\xi_1,\dots,x_n+i\xi_n)\]
so that $U(n)\subset \mathrm{Sp}(2n)$ and $\mathfrak{u}(n)\subset \mathfrak{sp}(2n)$.

\begin{lemma}\label{curv_calculation}
Let $T \in \mathfrak{u}(n)\subset \mathfrak{sp}(2n)$.
If $s\in \mcW$ is the vacuum projection of the harmonic oscillator $Q\in \mcW$, $Q(x,\xi)=\|x\|^2+\|\xi\|^2$, then
\[ \mu^{-1}(T)s = \frac{1}{2}\tr(T)s\]
\end{lemma}
\begin{proof}
First consider the harmonic oscillator in the case $n=1$. 
We have $iQ\in \mathfrak{g}$, and $\mu(iQ)=i\{iQ,\,\cdot\,\}=-\{Q,\,\cdot\,\}$. 
From  $\{Q,x\}=-2\xi$, $\{Q,\xi\}=2x$ we obtain
\[ \mu\left(\frac{i}{2}Q\right)=\begin{pmatrix} 0 &-1\\ 1 &0 \end{pmatrix}\]
Now let $V=\RR^{2n}=\CC^n$. 
Thinking of $T$ as an element in $\mathfrak{u}(n)$ acting on $\CC^n$,
choose an orthonormal basis  $e_1$, $e_2$,\ldots, $e_n$ of $\CC^n$ consisting of eigenvectors of $T$, with corresponding eigenvalues $i\beta_1,\dots, i\beta_n$.
The 2-dimensional real subspaces $V_k\subset \RR^{2n}$ spanned by $e_k, Je_k$ are $T$-invariant.
With respect to the decomposition $\RR^{2n}=V_1\oplus V_2 \oplus\ldots V_n$ we have $T= B_1\oplus B_2\oplus \ldots B_n$, where $B_k=\begin{pmatrix}0 &-\beta_k \\ \beta_k &0 \end{pmatrix}$.
Note that  $e_1, Je_1$, $e_2, Je_2$, \ldots, $e_n, Je_n$ is a symplectic basis of $V$.
If $y_1, \eta_1$, $y_2, \eta_2$, \ldots, $y_k, \eta_k \in V^*$ is the dual basis of $V^*$,
then the isomorphism $\mathfrak{sp}(V) \cong \mathfrak{sp}(V^*)$ is  the identity map of matrices. 

Let $Q_k:= y_k^2+\eta_k^2$.
Then $Q=\sum Q_k$ as a function on $\RR^{2n}$.
Under the isomorphism
\[ \mcW(\RR^{2n},\omega)=\bigotimes_{k=1}^n \mcW(V_k,dy_k\wedge d\eta_k)\]
we have
\[ Q = \sum_{k=1}^n 1\otimes\cdots \otimes Q_k\otimes \cdots \otimes 1\]
and
\[ s=s_1\otimes s_2\otimes \cdots\otimes s_n\]
where $s_k\in \mcW(V_k)$ is the vacuum projection of the harmonic oscillator $Q_k\in \mcW(V_k)$.
The above calculations show that 
\[
\iso^{-1}(T)= \sum \frac{i \beta_k}{2}\,1\otimes \cdots\otimes Q_k\otimes \cdots \otimes 1
\]
Since $Q_ks_k=s_k$ we get
\[ \mu^{-1}(T)s= \sum \frac{i \beta_k}{2}\,s_1\otimes \cdots\otimes s_k\otimes \cdots \otimes s_n=\frac{1}{2}\tr(T)s\]

\end{proof}

\begin{corollary}\label{halfc1}
Choose a complex structure $J$ on $H$ that is compatible with the symplectic structure.
Let $\nabla$ be a {\em unitary} connection on $H$ and let $\con$ and $\curv$ be as in \eqref{defcon} and \eqref{deftheta}.
Then 
\[ \Tr(\curv^is)=\left(\frac{1}{2}\tr(\theta)\right)^i\in \Omega^{2i}(M)\qquad i=0,1,2,3,\dots\]
\end{corollary}
\begin{proof}This follows immediately from Lemma \ref{curv_calculation},
and the fact that $s$ is a rank 1 projection.
\end{proof}

\subsection{The character of the symbol of a Toeplitz operator}

If $T_f$ is a Toeplitz operator acting on $H^2(M)\otimes \CC^r$, 
then  $\tilde{T}_f \colon =T_f\oplus (I-S\otimes I_r)$ is an operator on $L^2(M)\otimes \CC^r$
that has the same index as $T_f$.
$\tilde{T}_f$ is an order zero pseudodifferential operator in the Heisenberg calculus.
The principal Heisenberg symbol of $\tilde{T}_f$ is 
\[\sigma_H(\tilde{T}_f) = (fs+(1-s),1)\in M_r(\Symb)\]
More precisely, if 
\[ f = \sum_{i} f_i\otimes m_{i} \in C^\infty(M)\otimes M_r(\CC)\]
then 
\[\sigma_H(\tilde{T}_f) =  \sum_{i}  (f_{i}s,0)\otimes m_i+ (1-s,1)\otimes I_r\in  \Symb\otimes  M_r(\CC)\]

\begin{theorem}\label{character_Toeplitz}
Let  $M$ be a closed contact manifold, on which a generalized Szeg\"o projection has been chosen.
If $T_f$ is a Toeplitz operator on $M$ then
\[ \chi(\sigma_H(T_f)) =  \Ch(f)\wedge \exp(\frac{1}{2}c_1(H^{1,0}))\in H^{odd}(M)\]
\end{theorem}
\begin{proof}
Choose a unitary connection $\nabla$ for $H^{1,0}$.
We identify $H^{1,0}=H$ via $v\mapsto v+\bar{v}$, as usual.
If $\nabla$ is unitary, then  $\con s=0$.
Therefore the term $(1-s,1)\otimes I_r$ contributes zero to $\chi(\sigma_H(\tilde{T}_f))$.
Then  formula  \eqref{defchi} gives,
\begin{equation*}
    \chi(\sigma_H(T_f)) = \sum \limits_{l\ge 0} \sum \limits_{i_0, \ldots, i_{2l+1} \ge 0} \left(\frac{-1}{2 \pi i}\right)^{I+l+1} \frac{l!}{(I+2l+1)!} \Trh \left(\tr  r^{-1} \curv^{i_0} \con(
r)\curv^{i_1}\con(r^{-1}) \dots \con(r) \curv^{i_{2l+1}}\right)
\end{equation*}
where $I=i_0+i_1+\cdots+i_{2l+1}$ and we let
\[ r:=\sum_i f_is\otimes m_i \in \Symb\otimes M_r(\CC)\] 
From $\con s = 0$ we obtain
\[ \con(r)=\sum_i (df_i \otimes s)\otimes m_i \in \Omega^1(\Symb)\otimes M_r(\CC)\] 
or equivalently
\[ \con(r)=df \otimes s \in M_r(\Omega^1(M)) \otimes_{C^\infty(M)} \Symb\] 
Likewise
\[ \con(r^{-1})=d(f^{-1}) \otimes s\] 
Then 
\[r^{-1} \curv^{i_0} \con(r)\curv^{i_1}\con(r^{-1}) \dots \con(r) \curv^{i_{2l+1}}
=(-1)^l(f^{-1}df)^{2l+1}\otimes \curv^{i_0+\cdots+i_{2l+1}}s\]
Applying the traces we get
\[\Trh\left(\tr  r^{-1} \curv^{i_0} \con(r)\curv^{i_1}\con(r^{-1}) \dots \con(r) \curv^{i_{2l+1}} \right)=
(-1)^l\tr(f^{-1}df)^{2l+1}\cdot \Trh(\curv^{I}s)\]
By Corollary \ref{halfc1}, 
\[\Trh(\curv^{I}s)=\left(\frac{1}{2}\tr(\theta)\right)^{I}\]
Thus, $\chi(\sigma_H(T_f))$ can be written as a product of two factors,
\begin{equation*}
    \left(\sum \limits_{l\ge 0} (-1)^l\left(\frac{-1}{2 \pi i}\right)^{l+1}\frac{l!}{(2l+1)!}\tr(f^{-1}df)^{2l+1}\right)
    \left(\sum \limits_{i_0, \ldots, i_{2l+1} \ge 0} \left(\frac{-1}{2 \pi i}\right)^{I} \frac{(2l+1)!}{(I+2l+1)!} 
    \left(\frac{1}{2}\tr(\theta)\right)^{I}\right)
\end{equation*}
The first factor is $\chi(f)=\Ch(f)$ (see Proposition \ref{chern}).
For the second factor, note that 
\[ \frac{(I+2l+1)!}{I!(2l+1)!} \]
is the number of ways in which $I$
can be written as a sum $I=i_0+i_1+\cdots+i_{2l+1}$ of non-negative integers.
Therefore the second factor is
\[
 \sum_{k\ge 0}\frac{1}{k!} \left(\frac{-1}{2 \pi i}\cdot \frac{1}{2}\tr(\theta)\right)^{k}
    = \exp\left(\frac{1}{2}\cdot \frac{-1}{2 \pi i}\tr(\theta)\right)
\]
The first Chern class of $H^{1,0}$ is
\[ c_1(H^{1,0}) = \left[\frac{-1}{2 \pi i} \tr(\theta)\right]\in \Omega^2(M)\]
This completes the calculation.

\end{proof}

\section{$K$-theory}\label{section_Ktheory}

In this section we analyze the $K$-theory group $K_1(\Symb)$.
The aim of this section is to prove Proposition \ref{Ktheory}.
\vskip 6pt

For an order zero element $w\in \mcW^0$  the operator $\Op^w(w)$ is bounded on $L^2(\RR^n)$.
Let $W^0$ be the $C^*$-algebra that is the closure of $\mcW^0$ in the operator norm
(which is the same as the norm closure of $\mcW_q^0$). 
If $w\sim \sum w_{m}$ is the asymptotic expansion of $w=w(x,\xi)$,
then the leading term $w_0$ is constant on rays in $\RR^{2n}$,
and can be identified with a function $w_0\in C^\infty(S^{2n})$.
We have
\[ \|w_0\|_\infty \le \|\Op^w(w)\|\]
where $\|\cdot\|_\infty$ is the supremum norm. 
The map $w\mapsto w_0$ extends by continuity to a homomorphism of $C^*$-algebras
$W^0\to C(S^{2n-1})$.
The kernel of this map is the norm closure of the Schwartz space $\Sch(\RR^{2n})$,
which is the $C^*$-algebra $\KK$ of compact operators on $L^2(\RR)$.
More canonically, the kernel of $W^0(V,\omega)\to C(S(V))$  is the twisted convolution algebra $C^*(V,\omega)$
with product
\[ (f\ast_\omega g)(v) = \int e^{i\omega(v,w) } f(v-w)g(w)\,dw \qquad f,g\in C_c(V)\]
We have an $\mathrm{Sp}(V)$-equivariant short exact sequence,
\[ 0\to C^*(V,\omega) \to W^0(V,\omega)\to C(S(V))\to 0\]
For $a=(w_+,w_-)\in \Alg^0$ we define the $C^*$-norm
\[ \|a\| := \max\{ \|\Op^w(w_+)\|,  \|\Op^w(w_-)\|\}\]
Let $A^0$ be the $C^*$-algebra obtained by completing $\Alg^0$ in the norm $\|\cdot\|$.
If $a=(w_+,w_-)\in \Alg$ then the leading terms in the asymptotic expansions of $w_+$ and $w_-$ are equal.
We obtain a map
\[ \Alg^0\to C^\infty(S^{2n-1})\quad a=(w_+,w_-)\mapsto w_0\]
which  extends by continuity to $A^0\to C(S^{2n-1})$.
The kernel of this map is isomorphic to $\KK\oplus \KK$.
More canonically, we have a short exact sequence
\[ 0\to C^*(V,\omega)\oplus C^*(V,-\omega)\to  A^0(V,\omega)\to C(S(V))\to 0.\]
If $\sigma\in\Symb$,
then for $p\in M$ we have $\sigma(p)\in\Alg^0(H_p,\omega_p)$.
We let
\[ \|\sigma\| := \sup_{p\in M} \|\sigma(p)\|\]
Let $S_H$ be the $C^*$-algebraic closure of $\Symb$.
$S_H$ is the section algebra of a continuous field of $C^*$-algebras over $M$ whose fiber at $p\in M$ is  $A^0(H_p^*,\omega_p)$.
We obtain a short exact sequence of $C^*$-algebras
\[ 0\to I_H\to S_H\to C(S^*H)\to 0\]
In section 4 of \cite{BvE14} we analyze this short exact sequence.
The  ideal $I_H\subset S_H$ is the section algebra of a continuous field over $M$ with fibers $C^*(H_p^*,\omega_p)\oplus C^*(H_p^*,-\omega_p)\cong \KK\oplus\KK$.
In fact, $I_H$ is shown to be Morita equivalent to $C(M)\oplus C(M)$.

There is a $KK^M$-equivalence of short exact sequences,
\begin{equation*} \xymatrix {
 0\ar[r]
 & I_H\ar[r]\ar@{.>}[d]^{\cong}_{KK^M}
 & S_H\ar[r]\ar@{.>}[d]^{\cong}_{KK^M}
 & C(S^*H)\ar[r]\ar[d]^{=}  
 & 0
 \\
 0\ar[r]
 & C_0(H^*\sqcup H^*) \ar[r]
 & C(S^*M) \ar[r]
 & C(S^*H)\ar[r]
 &0
 }
 \end{equation*} 
 where $H^*_p\sqcup H^*_p$ is identified with the disjoint union of the (open) upper and lower hemispheres of $S^*_pM=S(H^*_p\times \RR)$.
From this $KK^M$-equivalence we obtain  a commutative diagram in $K$-theory,
\begin{equation*} \xymatrix {
K_1(I_H)\ar[r]\ar[d]^{\cong}
& K_1(S_H)\ar[r]\ar[d]^{\cong}
 & K^1(S^*H)\ar[d]^{=}  
 \\
K^1(H^*)\oplus K^1(H^*)\ar[r]
&K^1(S^*M) \ar[r]
 & K^1(S^*H)
 }
 \end{equation*} 
 The two rows are exact in the middle.
 
The isomorphism $K^1(H^*)\oplus K^1(H^*)\cong K_1(I_H)$ is obtained as the direct sum of two maps $K^1(H^*)\to K_1(I_H)$.
The first of these maps is as follows.
The symplectic  vector bundle $H^*$  is a Spin$^c$ vector bundle, and we have the Thom isomorphism
\[ K^1(H^*)\cong K^1(M)\]
For a continuous function $f \colon M\to U(r)$, the principal Heisenberg symbol of the Toeplitz operator $T_f$ is an element in $M_r(I_H)$.
We obtain a map
\[ K^1(M)\to K_1(I_H)\quad [f]\mapsto [\sigma_H(T_f)]\]
The first map $K^1(H^*)\to K_1(I_H)$ is the composition of the Thom isomorphism $K^1(H^*)\cong K^1(M)$
with the map $K^1(M)\to K_1(I_H)$ to symbols of Toeplitz operators.
(For  details see  \cite{BvE14}*{Section 4}.)

The second map $K^1(H^*)\to K_1(I_H)$ shall not concern us here.
It is similar to the first map, but constructed  using the opposite contact structure on $M$,
where $\alpha$ is replaced by $-\alpha$.

\begin{proposition}\label{Ktheory}
Let $M$ be a closed contact manifold.
Every element in $K_1(S_H)$ can be represented as the sum of the principal Heisenberg symbol of a Toeplitz operator $T_f$ for a smooth function $f \colon M\to U(r)$,  and an element in the image of the map $j_* \colon K^1(M)\to K_1(S_H)$ determined by the inclusion  $j\colon C(M)\hookrightarrow S_H$.
\end{proposition}

\begin{proof}

An element in $K^1(S^*M)$ is represented by a continuous map $S^*M\to U(r)$.
Every such map is homotopic to a smooth map $f \colon S^*M\to U(r)$ that, in each fiber, is constant on the lower hemisphere of $S^*_pM=S(H_p^*\oplus \RR)$.
The contact form $\alpha$ determines a map $-\alpha \colon M\to S^*M$ to the south pole in each fiber.
Let $g:=f\circ (-\alpha) \colon M\to U(r)$ and $\pi^*g=f\circ(-\alpha)\circ \pi \colon S^*M\to U(r)$, where $\pi \colon S^*M\to M$ is the projection.
Thus, $[f]+\pi^*[g] = [f(\pi^*g)]$ is in the image of the map $K^1(H^*)\to K^1(S^*M)$,
while $\pi^*[g]$ is the pullback of $[g]\in K^1(M)$ via $\pi$.

Now consider the isomorphism $K^1(S^*M)\cong K_1(S_H)$.
Commutativity in the diagram
\begin{equation*} \xymatrix {
K_1(C(M))\ar[r]^{j_*}\ar[d]^{\cong}
& K_1(S_H)\ar[d]^{\cong}
 \\
K^1(M)\ar[r]^{\pi^*}
&K^1(S^*M) }
 \end{equation*} 
implies that the element $\pi^*[g]\in K^1(S^*M)$ corresponds to $j_*[g]\in K_1(S_H)$.
By commutativity in the diagram
\begin{equation*} \xymatrix {
K_1(I_H)\ar[r]\ar[d]^{\cong}
& K_1(S_H)\ar[d]^{\cong}
 \\
K^1(H^*)\oplus K^1(H^*)\ar[r]
&K^1(S^*M)
 }
 \end{equation*} 
we conclude that every element in $K_1(S_H)$
is the sum of an element of the form $j_*[g]$ for $[g]\in K_1(C(M))$,
and an element in the composition the first map $K^1(H^*)\to K_1(I_H)$ with $K_1(I_H)\to K_1(S_H)$.
The latter element can be represented as the Heisenberg symbol of a Toeplitz operator, as explained above.

\end{proof}

\section{Proof of the index formula}\label{section_proof}

We are now ready to prove our index formula.

\begin{theorem}\label{main_theorem}
Let $M$ be a compact smooth manifold of dimension $2n+1$ with contact form $\alpha$. 
We orient  $M$ by the volume form $\alpha(d\alpha)^n$. 
If
\begin{equation*}P \colon C^\infty(M,\CC^r)\to C^\infty(M,\CC^r)\end{equation*}
is a Heisenberg pseudodifferential operator of  order zero
that acts on sections in a trivial bundle $M\times \CC^r$, 
with invertible Heisenberg principal  symbol $\sigma^0_H(P)\in M_r(\Symb)$, then the index of $P$ is
\begin{equation*} \ind  P = \int_M \chi(\sigma^0_H(P))\wedge \hat{A}(M).\end{equation*}
Here $\chi$ is the character homomorphism from Theorem \ref{summarychi}.
\end{theorem}
\begin{proof}
The principal Heisenberg symbol $\sigma^0_H(P)\in M_r(\Symb)$ determines an element in $K$-theory,
\[ [\sigma^0_H(P)]\in K_1(\Symb)\]
The index determines a homomorphism,
\[ K_1(\Symb)\to \ZZ\quad [\sigma_H^0(P)]\mapsto \ind  P\]
We need to verify that this index map is equal to the map 
\[ K_1(\Symb)\to \ZZ\quad \sigma \mapsto \int_M \chi(\sigma)\wedge \hat{A}(M)\]
Since $\Symb$ is a holomorphically closed dense subalgebra of the $C^*$-algebra $S_H$ (see the appendix),
we have
\[ K_1(\Symb)\cong K_1(S_H)\]
By Proposition \ref{Ktheory},
it suffices to verify the index formula in the case where $P$ is a Toeplitz operator
or $P$ is a  vector bundle automorphism.
The index of a vector bundle automorphism is zero, and in this case Theorem \ref{main_theorem} follows from Lemma \ref{vb_automorphism} below.

By Theorem \ref{character_Toeplitz}, for a Toeplitz operator $T_f$,
\[\int_M \chi(\sigma^0_H(T_f))\wedge \hat{A}(M) = \int_M \Ch(f)\wedge \exp\left(\frac{1}{2}c_1(H^{1,0})\right)\wedge \hat{A}(M)\]
Note that 
\[ \Td(H^{1,0}) = \exp\left(\frac{1}{2}c_1(H^{1,0})\right)\wedge \hat{A}(M)\]
Thus, in the case of order zero operators, Theorem \ref{main_theorem} reduces
to  Boutet de Monvel's theorem, and its generalization to contact manifolds due to Epstein and Melrose (Theorems \ref{BdM} and  \ref{EM_Toeplitz}).

\end{proof}

If $g$ is an automorphism of the trivial vector bundle $M\times \CC^r$,
then $g$ is an order zero operator in the Heisenberg calculus.
Its principal Heisenberg symbol is the image of $g\in \mathrm{GL}(r,C^\infty(M))$ in $\mathrm{GL}(r, \Symb)$
via the inclusion $C^\infty(M)\subset \Symb$.

\begin{lemma}\label{vb_automorphism}
If $g\in \mathrm{GL}(r,C^\infty(M))\subset \mathrm{GL}(r,\Symb)$ then $\chi(g)=0$.
\end{lemma}
\begin{proof}
Since $\con(g) = dg$ commutes with $\curv$, we have
 \begin{equation*}
\chi(g) = \sum \limits_{l\ge 0} \sum \limits_{i_0, \ldots, i_{2l+1} \ge 0} \left(\frac{-1}{2 \pi i}\right)^{I+l+1} \frac{l!}{(I+2l+1)!} \tau(\curv^I) \tr  g^{-1} dg\,d(g^{-1})\cdots dg,
\end{equation*}
where $I=i_0+\cdots+i_{2l+1}$.
By Lemma  \ref{tau_of_curv}, $\tau(\curv^I)=0$ for all $I$.

\end{proof}

\begin{remark}
The index formula of Theorem \ref{main_theorem} also applies verbatim to operators of arbitrary integer order $m\in\ZZ$.
Let $\Delta_H$ be a sublaplacian on $M$.
Then 
\[ (1+\Delta_H)^{-m/2}\circ P  \colon C^\infty(M,\CC^r)\to C^\infty(M,\CC^r)\]
is a Heisenberg elliptic operator of order zero, with the same index as $\msL$.
The Heisenberg principal symbol of $\Delta_H$ is $\sigma^2_H(\Delta_H)=(Q,Q)$,
where $Q$ is the harmonic oscillator (see section \ref{sec:symbols}). 
If  $\sigma_H^m(P)=(\sigma_+,\sigma_-)$ then
\[ \sigma_H^0((1+\Delta_H)^{-m/2}\circ P) = (Q^{-m/2}\#\sigma_+,\sigma_-\#Q^{-m/2})\in \Alg_H\]
If we choose a unitary connection on $H$,
then $\con (Q,Q)=0$ and $(Q,Q)$ commutes with $\curv$ in $\mcW_H\oplus \mcW_H^{op}$.
An easy calculation shows  that
\[ \chi(Q^{-m/2}\#\sigma_+,\sigma_-\#Q^{-m/2})=\chi(\sigma_+,\sigma_-)\]

\end{remark}

\begin{remark}
An alternative proof of Theorem \ref{main_theorem}, which gives some insight into  the nature of $\chi(\sigma)$,
is as follows.
Let $\Phi$ denote the  isomorphism
\[ \Phi:K_1(\Symb)\stackrel{\cong}{\lra} K^1(S^*M)\]
of section \ref{section_Ktheory}.
As shown in \cite{vE10a},
the index of a Heisenberg elliptic operator $P$
is computed by the Atiyah-Singer index formula applied to $\Phi(\sigma_H(P))$,
\[ \ind P = \int_{S^*M} \Ch(\Phi(\sigma_H(P)))\wedge \Td(TM\otimes \CC)\]
The question left open in \cite{vE10a} was how to compute $\Phi(\sigma_H(P))$.
Our index formula can  be interpreted as precisely such a  computation.
\begin{proposition}
On a compact contact manifold $M$,
\[ \int_M\chi(\sigma)\wedge \hat{A}(M) = \int_{S^*M}\Ch(\Phi(\sigma))\wedge \Td(TM\otimes \CC)\]
for all $\sigma\in K_1(\Symb)$.
\end{proposition}
\begin{proof}
As in the proof of Theorem \ref{main_theorem}, it suffices to prove the equality for two kinds of elements in $K_1(\Symb_1)$.
Let $\pi:TM\to M$ be the projection and $\pi_!:H_c^{k+2n+1}(TM)\to H^k(M)$
denotes integration in the fiber.
For a smooth map $g:M\to U(r)$ we have $\pi_!(\pi^*g)=0$, and the equality follows from Lemma \ref{vb_automorphism}.
If $f:M\to U(r)$ and 
\[\sigma_f=(fs+(1-s),1)\in \Symb\]
 is the symbol of a Toeplitz operator,
then by Theorem \ref{character_Toeplitz},
\[ \chi(\sigma_f)\wedge \hat{A}(M)=\Ch(f)\wedge \Td(H^{1,0})\]
From the results discussed in section \ref{section_Ktheory} one derives,
\[ \Phi(\sigma_f)=f\otimes \Lambda\in K^1(S^*M)\]
where  $f\in K^1(M)$ and $\Lambda\in K^0(S^*M)$ is the Thom class vector bundle of $S^*M=S(H^{0,1}\oplus \underline{\RR})$
($\Lambda$  restricts to a Bott generator vector bundle in each fiber $\cong S^{2n}$).
A characteristic class calculation shows that
\[ \pi_!(\Ch(\Lambda))=\left(\Td(H^{0,1})\right)^{-1}\]
The proposition then follows from
\[ \Td(TM\otimes \CC) = \Td(H^{1,0})\wedge \Td(H^{0,1})\]

\end{proof}
Note that since $\Td(TM\otimes \CC)=\hat{A}(M)^2$, we may state the same result as
\[ [\chi(\sigma)] = \pi_!(\Ch(\Phi(\sigma)))\wedge \hat{A}(M)\in H^{odd}(M)\]
for all $\sigma\in K_1(\Symb)$.

\end{remark}

\section*{Appendix. Holomorphic closure of the Fr\'echet algebra $\Symb$}

The algebra of Heisenberg principal symbols of order zero $\Symb$
is a Fr\'echet algebra. It is a holomorphically closed dense subalgebra of the $C^*$-algebra $S_H$,
and therefore the inclusion $\Symb\to S_H$ determines an isomorphism in topological $K$-theory
\[ K_1(\Symb)\cong K_1(S_H)\]
The proof of these facts follows standard techniques. 
For the convenience of the reader,  we sketch  the steps in this appendix. A good resource on many of these topics is \cite{IML}.
See also \cite{CGGP92}*{Section 5}.

\vskip 6pt
The continuous seminorms for the Fr\'echet topology of the Weyl algebra $\mcW^0$ are  the optimal constants in the inequalities \eqref{seminorms1} and \eqref{seminorms2}, together with the  $C^\infty$ seminorms of the restrictions of the terms $a_j$ of the asymptotic expansion \eqref{asymptotic expansion}.  

$\mcW^0$ is a  Fr\'echet algebra with this topology, i.e. the product $\#$ is continuous.
Indeed, if 
\[c=a\# b\qquad c\sim c_0+c_{1}+ \ldots\quad a \sim a_0+a_1+\ldots\quad b\sim b_0+b_1 +\ldots\]
then $c_k$ is obtained by applying a bidifferential operator  to $a_i$, $b_j$, $i, j \le k$. It follows that one can bound the $C^\infty$-norms of $c_k$ by the  $C^\infty$ norms of $a_i$, $b_j$. The seminorms of \eqref{seminorms1}, \eqref{seminorms2} can be bound  with the help of estimates for the remainder term in the Weyl product (See \cite{HormIII}*{Theorem 18.5.4}). 

$\mcW_q^0$ is a closed subalgebra of $\mcW^0$, and $\Alg^0$ in turn is a closed subalgebra of  $\mcW_q^0\oplus \left(\mcW_q^0\right)^{op}$. 

\vskip 6pt

The algebra of Schwarz functions $\mcS= \mcS(\mathbb{R}^{2n})$ is a closed two-sided ideal in $\mcW^0$.
Let $\mathcal{I}:=\Op^w(\mcS)$ be the algebra of smoothing operators 
in the Weyl calculus.
The norm closure of $\mcI$ in $\mcB(L^2(\RR^n))$ is the ideal $\KK=\KK(L^2(\RR^n))$ of compact operators.
As a first step, we show that the unitalization $\mcI^+$ is  holomorphically closed in $\KK^+$.

$\mathcal{I}$ itself is not an ideal in $\mathcal{B}(L^2(\RR^n))$. However 
\begin{equation}\label{biideal}
 \forall A_1, A_2 \in \mathcal{I}, B \in \mathcal{B}(L^2(\RR^n)),\ A_1BA_2 \in \mathcal{I}.   
\end{equation}
This follows from the following characterisation of $\mathcal{I}$: An operator $A\in \mcB(L^2(\RR^n))$ is in $\mathcal{I}$ if and only if for every $k$, $l$ the operator $\mathcal{H}^k A \mathcal{H}^l$, defined on the space of Schwarz functions $\mcS(\mathbb{R}^n)$, extends to a bounded operator on $L^2(\mathbb{R}^n)$. Here $\mathcal{H}$ is the harmonic oscillator on $L^2(\RR^n)$.

\begin{lemma}\label{s holomorphically closed}
Let $A\in \mathcal{I} $, so that $1+A$ is a bounded Fredholm operator. Let $P_0$ be the orthogonal projection on $\ker (1+A)$, and $P_1$ the orthogonal projection on $\ker (1+A^*)$.
Choose a parametrix  $1+B \in \mathcal{B}(L^2(\RR^n))$ of  $1+A$ such that
\[(1+B)(1+A)=1-P_0\quad  (1+A)(1+B) =1- P_1\]
Then $B \in \mathcal{I}$.
\end{lemma}
\begin{proof}
If $(1+A)f=0$ then $f=-Af\in \Sch(\RR^n)$.
If $f_1, f_2, \dots, f_k$ is an orthonormal basis for the kernel of $1+A$, 
then the Schwartz kernel of $P_0$ is $\sum f_j(x)\bar{f_j(y)}\in \Sch(\RR^{2n})$, and so $P_0\in \mcI$.
Likewise  $P_1\in \mcI$.

From $A+B+BA=-P_0$ we get $A^2+AB+ABA=-AP_0$.
Since $ABA\in \mcI$ by \eqref{biideal} we have $AB\in \mcI$. 
Then $A+B+AB=-P_1$ implies $B\in \mcI$.

\end{proof}

\begin{corollary}\label{Iplus}
If  $A \in \mathcal{I}^+$ is invertible  as a bounded operator  on $L^2(\RR^n)$, then $A^{-1}\in \mathcal{I}^+$.
\end{corollary}
Let $W^0$ be the $C^*$-algebra that is the norm closure of $\mcW^0\subset \mcB(L^2(\RR^n))$.
Lemma \ref{s holomorphically closed} together with the Weyl symbolic calculus implies  that $\mcW^0$ is holomorphically closed in $W^0$.

\begin{lemma}
If $a\in \mcW^0$ is such that $A=\Op^w(a)$ is invertible as a bounded operator on  $L^2(\RR^n)$, then $A^{-1} = \Op^w(b)$ for some $b\in \mcW^0$.
\end{lemma}
\begin{proof}
We have the short exact sequence
\[ 0\to \KK(L^2(\RR^n))\to W^0\stackrel{\sigma_W}{\lra} C(S^{2n-1})\to 0\]
If $\Op^w(a)$ is invertible in $\mcB(L^2(\RR^n))$, it is invertible in $W^0$.
Then $\sigma_W(a)$ is invertible in $C(S^{2n-1})$.
The principal Weyl symbol $\sigma_W(a)$ is the leading term  $\sigma_W(a) = a_0 \in C^\infty(S^{2n-1})$ in the asymptotic expansion $a\sim a_0+a_1+\cdots$.
Since $\sigma_W(a)$ is invertible in $C(S^{2n-1})$, it is also invertible in $C^\infty(S^{2n-1})$.
In the Weyl symbolic calculus, using \eqref{eqn:asymptotic}, this implies that the full Weyl symbol $a_0+a_1+\cdots$ is invertible.
Then  there exists $r \in \mcW^0$ such that  
\[ 1-a\#r =q_0 \in \mcS\qquad 1-r\#a=q_1 \in \mcS\]
Denote $R:=\Op^w(r)$, $Q_0:=\Op^w(q_0)$, $Q_1:=\Op^w(q_1)$, so that
\[ 1-AR =Q_0 \in \mcI\qquad 1-RA=Q_1 \in \mcI\]
By Lemma \ref{s holomorphically closed} there exists $T \in \mcI$ such that $(1-Q_0)(1+T) =1-P_0$, where $P_0\in \mcI$ is a projection. 
Then $S:=R(1+T)$ is another parametrix for $A$. We have $1-AS=P_0$, while $1-SA=Q_1-RTA\in \mcI$. Let $P_1:=Q_1-RTA$. ($P_1$ is not necessarily a projection.)
Then  $S=A^{-1}-A^{-1}P_0 = A^{-1}-P_1A^{-1}$. Hence $A^{-1}P_0=P_1A^{-1}$. 
Since $P_0$ is an idempotent, $A^{-1}P_0=P_1A^{-1}P_0$.
By \eqref{biideal} we have $P_1A^{-1}P_0 \in \mathcal{I}$, and so $A^{-1}-S =A^{-1}P_0\in \mathcal{I}$.
Therefore $A^{-1} =\Op^w(b)$ for some $b \in \mcW^0$.

\end{proof}
\begin{corollary}
If $a \in \mcW_q^0$ is invertible in $W^0$, then $a^{-1} \in \mcW_q^0$. Similarly,  if $a \in \Alg^0$ is invertible in $A^0$ then $a^{-1} \in \mcA^0$.

\end{corollary}

We will now discuss smooth dependence of inverses on parameters. 
We will use the following  version of \eqref{biideal}, which can be proved in a similar manner:
If $a_1(s)$, $a_2(s)$ are smooth families in $\mcS$ depending on a parameter $s \in \mathbb{R}^m$, and $B(s) \in \mathcal{B}(L^2(\RR^n))$ is a  smooth family of operators with respect to the operator norm, then
there exists a smooth family $c(s)$ of elements in $\mcS$ with
\[\Op^w(a_1(s)) B(s) \Op^w(a_2(s)) =\Op^w(c(s))\]


\begin{lemma}
Let $1+q(s)$, $s \in \mathbb{R}^m$, $q(s) \in \mcS$, be a smooth invertible family. Then $(1+q(s))^{-1} =1+p(s)$, where $p(s)$ is also a smooth family in $\mcS$.
\end{lemma}
\begin{proof}
Let $Q(s) = \Op^w(q(s))$. Note that $(1+Q(s))^{-1}$ is a smooth family of operators in $\mathcal{B}(L^2(\RR^n))$. Then
\[
(1+Q(s))^{-1}= 1-Q(s)+Q(s) (1+Q(s))^{-1} Q(s).
\]
Hence $p(s)= -q(s)+c(s)$ where $\Op^w(c(s))=Q(s) (1+Q(s))^{-1} Q(s)$.

\end{proof}
 
\begin{proposition}   Assume  that $a(s) \in \mcW^0$, $s\in \mathbb{R}^m$ is a smooth family (in the sense that $a(x,\xi, s)$ is a smooth function on the radial compactification of $\mathbb{R}^{2n}$ times $\mathbb{R}^m$) such that $a(0)$ is invertible. 
Then for $s$ close to $0$ there exists a smooth family $b(s)$ such that $\Op^w(a(s))^{-1}=\Op^w(b(s))$.
\end{proposition}
\begin{proof}
First one constructs a smooth family $c(s)$  such that 
$c(0)=a(0)^{-1}$, $1-a(s)\#c(s) =q_0(s) \in \mathcal{S}$, $1-c(s)\#a(s)=q_1(s) \in \mathcal{S}$.
Construct an arbitrary smooth family of parametrices for $a(s)$ and add a constant in $s$ term in $\mathcal{I}$ to achieve $c(0)=a(0)^{-1}$. Note that $q_i(s)$ is smooth in $s$ and $q_i(0)=0$. Hence $\| \Op^w(q_i(s))\| <1$ for small $s$ and $1-\Op^w(q_i(s))$ is an invertible operator. Therefore there exists a smooth family $p(s) \in \mathcal{S}$ such that $(1+p(s))\#(1-q_1(s))=1$. 
If we set $b(s)=(1+p(s))\#(1-q_1(s))$, we have $b(s)\# a(s)=1$ for small $s$. Since $a(s)$ is invertible for small $s$, $b(s)$ is the inverse.

\end{proof}

\begin{corollary}
If $\sigma \in \Symb$  is invertible in $S_H$ then  $\sigma^{-1} \in \Symb$.
\end{corollary}
For the Frechet algebra $\Symb$, the topological $K$-theory is defined as
\[K_1(\Symb):= \pi_0(GL(\Symb))\]
where  $GL(\Symb)$ is the topological group obtained as the direct limit $\varinjlim GL(r,\Symb)$.
We  have (see \cite{Co94}*{III. Appendix C}):
\begin{corollary}
The inclusion $\Symb \to S_H$ induces an isomorphism in topological $K$-theory,
\[
K_1(\Symb) \cong K_1(S_H).
\]

\end{corollary}

\bibliographystyle{amsplain}

\bibliography{MyBibfile}

\begin{bibdiv}
\begin{biblist}

\bib{BvE14}{article}{
      author={Baum, Paul~F.},
      author={van Erp, Erik},
       title={{$K$}-homology and index theory on contact manifolds},
        date={2014},
        ISSN={0001-5962},
     journal={Acta Math.},
      volume={213},
      number={1},
       pages={1\ndash 48},
         url={https://doi.org/10.1007/s11511-014-0114-5},
      review={\MR{3261009}},
}

\bib{BG88}{book}{
      author={Beals, Richard},
      author={Greiner, Peter},
       title={Calculus on {H}eisenberg manifolds},
      series={Annals of Mathematics Studies},
   publisher={Princeton University Press},
     address={Princeton, NJ},
        date={1988},
      volume={119},
}

\bib{Bo79}{article}{
      author={Boutet~de Monvel, Louis},
       title={On the index of {T}oeplitz operators of several complex
  variables},
        date={1978/79},
        ISSN={0020-9910},
     journal={Invent. Math.},
      volume={50},
      number={3},
       pages={249\ndash 272},
         url={http://dx.doi.org/10.1007/BF01410080},
}

\bib{CGGP92}{article}{
      author={Christ, Michael},
      author={Geller, Daryl},
      author={G{\l}owacki, Pawe{\l}},
      author={Polin, Larry},
       title={Pseudodifferential operators on groups with dilations},
        date={1992},
        ISSN={0012-7094},
     journal={Duke Math. J.},
      volume={68},
      number={1},
       pages={31\ndash 65},
         url={http://dx.doi.org/10.1215/S0012-7094-92-06802-5},
}

\bib{Co85}{article}{
      author={Connes, Alain},
       title={Noncommutative differential geometry},
        date={1985},
        ISSN={0073-8301},
     journal={Inst. Hautes \'{E}tudes Sci. Publ. Math.},
      number={62},
       pages={257\ndash 360},
  url={http://www.numdam.org.dartmouth.idm.oclc.org/item?id=PMIHES_1985__62__257_0},
      review={\MR{823176}},
}

\bib{Co94}{book}{
      author={Connes, Alain},
       title={Noncommutative geometry},
   publisher={Academic Press Inc.},
     address={San Diego, CA},
        date={1994},
        ISBN={0-12-185860-X},
}

\bib{EMxx}{book}{
      author={Epstein, Charles},
      author={Melrose, Richard},
       title={The {H}eisenberg algebra, index theory and homology},
        note={Unpublished manuscript},
}

\bib{EM98}{article}{
      author={Epstein, Charles},
      author={Melrose, Richard},
       title={Contact degree and the index of {F}ourier integral operators},
        date={1998},
        ISSN={1073-2780},
     journal={Math. Res. Lett.},
      volume={5},
      number={3},
       pages={363\ndash 381},
  url={https://doi-org.dartmouth.idm.oclc.org/10.4310/MRL.1998.v5.n3.a9},
      review={\MR{1637844}},
}

\bib{Ep04}{incollection}{
      author={Epstein, Charles~L.},
       title={Lectures on indices and relative indices on contact and
  {CR}-manifolds},
        date={2004},
   booktitle={Woods {H}ole mathematics},
      series={Ser. Knots Everything},
      volume={34},
   publisher={World Sci. Publ., Hackensack, NJ},
       pages={27\ndash 93},
  url={https://doi-org.dartmouth.idm.oclc.org/10.1142/9789812701398_0002},
      review={\MR{2123367}},
}

\bib{FS74}{article}{
      author={Folland, G.~B.},
      author={Stein, E.~M.},
       title={Estimates for the {$\bar \partial _{b}$} complex and analysis on
  the {H}eisenberg group},
        date={1974},
        ISSN={0010-3640},
     journal={Comm. Pure Appl. Math.},
      volume={27},
       pages={429\ndash 522},
}

\bib{gor1}{article}{
      author={Gorokhovsky, Alexander},
       title={Characters of cycles, equivariant characteristic classes and
  {F}redholm modules},
        date={1999},
        ISSN={0010-3616},
     journal={Comm. Math. Phys.},
      volume={208},
      number={1},
       pages={1\ndash 23},
         url={https://doi.org/10.1007/s002200050745},
      review={\MR{1729875}},
}

\bib{gor2}{book}{
      author={Gorokhovsky, Alexander},
       title={Explicit formulae for characteristic classes in noncummutative
  geometry},
        date={1999},
        ISBN={978-0599-42874-4},
        note={Thesis (Ph.D.)--The Ohio State University},
      review={\MR{2699561}},
}

\bib{GLS68}{article}{
      author={Grossmann, A.},
      author={Loupias, G.},
      author={Stein, E.~M.},
       title={An algebra of pseudodifferential operators and quantum mechanics
  in phase space},
        date={1968},
        ISSN={0373-0956},
     journal={Ann. Inst. Fourier (Grenoble)},
      volume={18},
      number={fasc., fasc. 2},
       pages={343\ndash 368, viii (1969)},
  url={http://www.numdam.org.dartmouth.idm.oclc.org/item?id=AIF_1968__18_2_343_0},
      review={\MR{267425}},
}

\bib{guil}{article}{
      author={Guillemin, Victor},
       title={A new proof of {W}eyl's formula on the asymptotic distribution of
  eigenvalues},
        date={1985},
        ISSN={0001-8708},
     journal={Adv. in Math.},
      volume={55},
      number={2},
       pages={131\ndash 160},
         url={https://doi.org/10.1016/0001-8708(85)90018-0},
      review={\MR{772612}},
}

\bib{Ho79}{article}{
      author={H\"{o}rmander, L.},
       title={The {W}eyl calculus of pseudodifferential operators},
        date={1979},
        ISSN={0010-3640},
     journal={Comm. Pure Appl. Math.},
      volume={32},
      number={3},
       pages={360\ndash 444},
         url={https://doi-org.dartmouth.idm.oclc.org/10.1002/cpa.3160320304},
      review={\MR{517939}},
}

\bib{HormIII}{book}{
      author={H\"{o}rmander, Lars},
       title={The analysis of linear partial differential operators. {III}},
      series={Grundlehren der Mathematischen Wissenschaften [Fundamental
  Principles of Mathematical Sciences]},
   publisher={Springer-Verlag, Berlin},
        date={1994},
      volume={274},
        ISBN={3-540-13828-5},
        note={Pseudo-differential operators, Corrected reprint of the 1985
  original},
      review={\MR{1313500}},
}

\bib{loday}{book}{
      author={Loday, Jean-Louis},
       title={Cyclic homology},
      series={Grundlehren der Mathematischen Wissenschaften [Fundamental
  Principles of Mathematical Sciences]},
   publisher={Springer-Verlag, Berlin},
        date={1992},
      volume={301},
        ISBN={3-540-53339-7},
  url={https://doi-org.dartmouth.idm.oclc.org/10.1007/978-3-662-21739-9},
        note={Appendix E by Mar\'{\i}a O. Ronco},
      review={\MR{1217970}},
}

\bib{IML}{unpublished}{
      author={Melrose, Richard},
       title={Introduction to microlocal analysis},
        note={2007 lecture notes, \url{https://math.mit.edu/~rbm/iml/}.},
}

\bib{nt2}{article}{
      author={Nest, Ryszard},
      author={Tsygan, Boris},
       title={Algebraic index theorem},
        date={1995},
        ISSN={0010-3616},
     journal={Comm. Math. Phys.},
      volume={172},
      number={2},
       pages={223\ndash 262},
         url={http://projecteuclid.org/euclid.cmp/1104274104},
      review={\MR{1350407}},
}

\bib{nt3}{article}{
      author={Nest, Ryszard},
      author={Tsygan, Boris},
       title={Algebraic index theorem for families},
        date={1995},
        ISSN={0001-8708},
     journal={Adv. Math.},
      volume={113},
      number={2},
       pages={151\ndash 205},
         url={https://doi.org/10.1006/aima.1995.1037},
      review={\MR{1337107}},
}

\bib{nt1}{incollection}{
      author={Nest, Ryszard},
      author={Tsygan, Boris},
       title={On the cohomology ring of an algebra},
        date={1999},
   booktitle={Advances in geometry},
      series={Progr. Math.},
      volume={172},
   publisher={Birkh\"{a}user Boston, Boston, MA},
       pages={337\ndash 370},
      review={\MR{1667686}},
}

\bib{Po01}{article}{
      author={Ponge, Rapha{\"e}l},
       title={G\'eom\'etrie spectrale et formules d'indices locales pour les
  vari\'et\'es {CR} et contact},
        date={2001},
        ISSN={0764-4442},
     journal={C. R. Acad. Sci. Paris S\'er. I Math.},
      volume={332},
      number={8},
       pages={735\ndash 738},
}

\bib{Sh01}{book}{
      author={Shubin, M.~A.},
       title={Pseudodifferential operators and spectral theory},
     edition={Second},
   publisher={Springer-Verlag, Berlin},
        date={2001},
        ISBN={3-540-41195-X},
  url={https://doi-org.dartmouth.idm.oclc.org/10.1007/978-3-642-56579-3},
        note={Translated from the 1978 Russian original by Stig I. Andersson},
      review={\MR{1852334}},
}

\bib{Ta84}{article}{
      author={Taylor, Michael~E.},
       title={Noncommutative microlocal analysis. {I}},
        date={1984},
        ISSN={0065-9266},
     journal={Mem. Amer. Math. Soc.},
      volume={52},
      number={313},
       pages={iv+182},
}

\bib{vE10a}{article}{
      author={van Erp, Erik},
       title={The {A}tiyah-{S}inger formula for subelliptic operators on a
  contact manifold, {P}art {I}},
        date={2010},
     journal={Ann. of Math.},
      number={171},
       pages={1647\ndash 1681},
}

\bib{vE10b}{article}{
      author={van Erp, Erik},
       title={The {A}tiyah-{S}inger formula for subelliptic operators on a
  contact manifold, {P}art {II}},
        date={2010},
     journal={Ann. of Math.},
      number={171},
       pages={1683\ndash 1706},
}

\end{biblist}
\end{bibdiv}

\end{document}